\documentclass[twoside]{article}

\usepackage[T1]{fontenc}
\usepackage[french,english]{babel}
 
\let\cedille\c
\usepackage[accepted]{aistats2019}
\usepackage[utf8x]{inputenc}
\usepackage{mystyle}
\usepackage{notations}
\usepackage{url}

\usepackage{pgffor}
\usepackage{pgfplots}
\usepackage{tikz,tkz-euclide}
\usepackage{tikz-3dplot}
\usetikzlibrary{shapes.multipart,calc}
\usetikzlibrary{decorations.markings}
\usetikzlibrary{arrows.meta}
\usepackage[export]{adjustbox}
\usepackage{rotating} 
\usepackage[round]{natbib}

\begin{document}

%

%
\runningauthor{Feydy, S\'ejourn\'e, Vialard, Amari, Trouv\'e, Peyr\'e}
\runningtitle{Interpolating between Optimal Transport and MMD using Sinkhorn Divergences}

\twocolumn[

\vspace{-\baselineskip}
\aistatstitle{Interpolating between Optimal Transport and MMD\\ using Sinkhorn Divergences}


\vspace{-\baselineskip}
\aistatsauthor{ Jean Feydy \And Thibault S\'ejourn\'e \And  Fran\cedille{c}ois-Xavier Vialard}
\aistatsaddress{ DMA, \'Ecole Normale Sup\'erieure\\CMLA, ENS Paris-Saclay  \And DMA, \'Ecole Normale Sup\'erieure \And LIGM, UPEM } 

\vspace{-\baselineskip}
\aistatsauthor{  Shun-ichi Amari \And Alain Trouv\'e \And Gabriel Peyr\'e }
\aistatsaddress{  Brain Science Institute, RIKEN  \And CMLA, ENS Paris-Saclay \And  DMA, \'Ecole Normale Sup\'erieure  } 

]


\begin{abstract}
  	Comparing probability distributions is a fundamental problem in data sciences. 
	Simple norms and divergences such as the total variation and the relative entropy only compare densities in a point-wise manner and fail to capture the geometric nature of the problem.
	In sharp contrast, Maximum Mean Discrepancies (MMD) 
	and Optimal Transport distances (OT) 
	are two classes of distances between measures that take into account the geometry of the underlying space and metrize the convergence in law.

	This paper studies the Sinkhorn divergences, a family of geometric divergences that interpolates between MMD and OT.
	Relying on a new notion of geometric entropy, we provide theoretical guarantees for these divergences: positivity, convexity and metrization of the convergence in law.
	On the practical side, we detail a numerical scheme
	that enables the large scale application
	of these divergences for machine learning:
	on the GPU, gradients of the Sinkhorn loss
	can be computed for batches of a million samples.
	%
\end{abstract}

\section{Introduction}
\label{sec-intro}

Countless methods in machine learning and imaging sciences rely 
on comparisons between probability distributions.
With applications ranging from shape matching~\citep{vaillant2005surface,varifold}
to classification~\citep{2015-Frogner} and generative model 
training~\citep{goodfellow2014generative},
a common setting is that of \emph{measure fitting}:
given a unit-mass, positive empirical distribution $\beta \in \Mm_1^+(\Xx)$
on a feature space $\Xx$,
a loss function 
$\L : \Mm_1^+(\Xx)\times \Mm_1^+(\Xx)\rightarrow \R$
and a model distribution $\al_\theta \in\Mm_1^+(\Xx)$
parameterized by a vector $\theta$,
we strive to minimize $\theta\mapsto \L(\al_\theta,\be)$
through gradient descent.
Numerous papers focus on the construction of
suitable models $\theta\mapsto \al_\theta$.
But which loss function $\L$ should we use?
If $\Xx$ is endowed with a \emph{ground distance}
$\d:\Xx\times\Xx\rightarrow \R_+$, taking
it into account can make sense and help
descent algorithm to overcome spurious
local minima.


\paragraph{Geometric divergences for Machine Learning.}

Unfortunately, simple dissimilarities  such as the Total Variation norm  or the Kullback-Leibler relative entropy do not take into account the distance~$\d$ on the feature space~$\Xx$. As a result, they do not metrize the convergence in law (aka. the weak$^*$ topology of measures) and are unstable with respect to deformations of the distributions' supports. 
We recall that if $\Xx$ is compact,
$\al_n$ converges weak$^*$ towards $\al$ (denoted $\al_n \rightharpoonup \al$) if for all continuous test functions $f\in \Cc(\Xx)$, $\dotp{\al_n}{f} \rightarrow \dotp{\al}{f}$ where $\dotp{\al}{f} \eqdef \int_\Xx f \d\al = \EE[f(X)]$ for any random vector $X$ with law $\al$.

The two main classes of losses $\L(\al,\be)$ which avoid these shortcomings are Optimal Transport distances and Maximum Mean Discrepancies: they are continuous with respect to the convergence in law and metrize its topology.
That is, $\al_n \rightharpoonup \al \Leftrightarrow \L(\al_n,\al) \rightarrow 0$.
The main purpose of this paper is to study the theoretical properties of a new class of \emph{geometric} divergences which interpolates between these two families and thus offers an extra degree of freedom through a parameter $\epsilon$ that can be cross-validated in typical learning scenarios.

\subsection{Previous works}

\paragraph{OT distances and entropic regularization.}

A first class of geometric distances between measures is that of Optimal Transportation (OT) costs,  which are computed as solutions of a linear program~\citep{Kantorovich42} (see~\eqref{eq-primal-ot} below in the special case $\epsilon=0$). Enjoying many theoretical properties, these costs allow us to lift a ``ground metric'' on the feature space $\Xx$ towards a metric on the space $\Mm_+^1(\Xx)$ of probability distributions \citep{santambrogio2015optimal}. 
%
OT distances (sometimes referred to as Earth Mover's Distances~\citep{rubner-2000}) are progressively being adopted 
as an effective tool in a wide range of situations,
from computer graphics~\citep{bonneel2016wasserstein} to supervised
learning~\citep{2015-Frogner}, unsupervised density
fitting~\citep{bassetti2006minimum} and generative model
learning
~\citep{montavon2016wasserstein,WassersteinGAN,
salimans2018improving,genevay2018learning,SanjabiSinkhGAN}. 
However, in practice, solving the linear problem required
to compute these OT distances is a challenging issue;
many algorithms that leverage the properties of the 
underlying feature space $(\Xx,\d)$ have thus 
been designed to accelerate the
computations, see~\citep{peyre2017computational} for an overview.

Out of this collection of methods, entropic regularization has recently emerged as a computationally efficient way 
of approximating OT costs.
For $\epsilon > 0$, we define
\begin{align}\label{eq-primal-ot}  \hspace*{-.5cm}
	\OT_\epsilon(\al,\be) \eqdef
	\!\!\umin{\pi_1=\al,\pi_2=\be} \!
	\int_{\Xx^2} \C \,\d\pi 
	+ \epsilon \KL(\pi|\al\otimes\be)
\end{align}
\eq{
    \qwhereq  \KL(\pi|\al\otimes\be) \eqdef \int_{\Xx^2} \log\pa{ \frac{\d\pi}{\d\al\d\be} } \d\pi, 
}
where $\C(x,y)$ is some symmetric positive cost function (we assume here that $\C(x,x)=0$) and 
where the minimization is performed over coupling measures $\pi \in \Mm_1^+(\Xx^2)$ as $(\pi_1,\pi_2)$ denotes the two marginals of $\pi$. 
Typically, $\C(x,y)= \|x-y\|^p$ on  $\Xx\subset\R^\D$ 
and setting $\epsilon=0$ in~\eqref{eq-primal-ot} allows us to retrieve the Earth Mover ($p=1$) or the quadratic Wasserstein ($p=2$) distances.

%
The idea of adding an entropic barrier $\KL(\,\cdot\,|\al\otimes\be)$ to the original linear OT program can be traced back to Schr\"odinger's problem~\citep{leonard2013survey} and has been used for instance in social sciences~\citep{galichon2010matching}.
Crucially, as highlighted in~\citep{CuturiSinkhorn}, the smooth problem~\eqref{eq-primal-ot}
can be solved efficiently on the GPU as soon as $\epsilon > 0$~:
the celebrated Sinkhorn algorithm (detailed in Section~\ref{sec:practice})
allows us to compute efficiently a smooth, geometric loss $\OT_\epsilon$ between sampled measures.

\paragraph{MMD norms.}

Still, to define geometry-aware distances between measures, a simpler approach
is to integrate a positive definite kernel $k(x,y)$ on the feature space $\Xx$. 
On a Euclidean feature space $\Xx\subset \R^\D$,
we typically use RBF kernels such as the Gaussian kernel $k(x,y)=\exp(-\norm{x-y}^2/2\si^2)$ or the energy distance (conditionally positive)
kernel $k(x,y)=-\norm{x-y}$. The kernel loss is then defined, for $\xi=\al-\be$, as
\begin{align}
  \label{def-mmd}
  \L_k(\al,\be)\eqdef
  \tfrac{1}{2}\norm{\xi}_k^2 \eqdef \tfrac{1}{2}\int_{\Xx^2} k(x,y)\,\d\xi(x)\d\xi(y).
\end{align}
If $k$ is universal~\citep{micchelli2006universal} (i.e. if the linear space spanned by functions $k(x,\cdot)$ is dense in $\Cc(\Xx)$) we know that $\norm{\cdot}_k$ metrizes the convergence in law. 
Such Euclidean norms, introduced for shape matching in~\citep{glaunes2004diffeomorphic}, are often referred to as ``Maximum Mean Discrepancies'' (MMD)~\citep{gretton2007kernel}.
They have been extensively used for generative model (GANs) fitting in machine learning~\citep{li2015generative,MMD-GAN}. 
MMD norms are cheaper to compute than OT and have a smaller \emph{sample complexity} -- i.e. approximation error when sampling a distribution. 
\subsection{Interpolating between OT and MMD using Sinkhorn divergences}

Unfortunately though, the ``flat'' geometry
that MMDs induce on the space of probability measures $\Mm_1^+(\Xx)$ does not faithfully lift the ground distance on $\Xx$. For instance, on $\Xx=\RR^\D$, let us denote by $\al_\tau$ the translation of $\al$ by $\tau \in \RR^\D$, defined through $\dotp{\al_\tau}{f} = \dotp{\al}{f(\cdot+\tau)}$ for continuous functions $f \in \Cc(\RR^\D)$.
Wasserstein distance discrepancies defined for $\C(x,y)=\norm{x-y}^p$ are such that $\OT_0(\al,\al_\tau)^{\frac{1}{p}} = \norm{\tau}$.

In sharp contrast, MMD norms rely on \emph{convolutions}
that weigh the frequencies of $\xi=\al-\be$ independently,
according to the Fourier transform of the kernel function.
For instance, up to a multiplicative constant of $\D$,
the energy distance 
$\L_{-\|\,\cdot\,\|}(\al,\al_\tau)$ 
is given by $\int_{\R^\D} (1-\cos(\langle \omega,\tau\rangle))~|\hat{\alpha}(\omega)|^2/\|\omega\|^{\D+1}d\omega$ 
 (\cite{SR05}, Lemma 1), where $\hat{\alpha}(\omega)$ is the Fourier transform of the probability measure $\alpha$.
Except for the trivial case of a Dirac mass $\al=\de_{x_0}$ for some $x_0 \in \RR^\D$, we thus always have $\L_{-\|\,\cdot\,\|}(\al,\al_\tau)<\norm{\tau}$ with
a value that strongly depends on the \emph{smoothness}
of the reference measure $\al$.
In practice, as evidenced in Figure~\ref{fig:flow2d}, this
theoretical shortcoming of MMD losses 
is reflected by \emph{vanishing gradients}
(or similar artifacts) next to the extreme points
of the measures' supports.


\textbf{Sinkhorn divergences.}
On the one hand, OT losses have appealing \emph{geometric} properties; on the other hand, cheap MMD norms scales up to large batches with a low sample complexity.
Why not \emph{interpolate} between them to get the best of both worlds?

Following~\citep{genevay2018learning} (see also~\citep{RamdasSinkhAsymptotics,salimans2018improving,SanjabiSinkhGAN}) we consider 
a new cost built from $\OT_\epsilon$ that we call a \emph{Sinkhorn divergence}:
\begin{equation} \label{eq-sinkhorn-div}
	\Wb_\epsilon(\al,\be) \!\eqdef\! \W_\epsilon(\al,\be) \!-\! \tfrac{1}{2}\W_\epsilon(\al,\al) \!-\! \tfrac{1}{2}\W_\epsilon(\be,\be).
\end{equation}
Such a formula satisfies $\Wb_\epsilon(\beta,\beta) = 0$ and interpolates between OT and MMD~\citep{RamdasSinkhAsymptotics}:
\eql{ \label{eq:sinkhorn-interpolation}
	\OT_0(\al,\be) 
	\:\overset{0 \leftarrow \epsilon}{\longleftarrow}\:
	\Wb_\epsilon(\al,\be) 
	\:\overset{\epsilon\rightarrow +\infty}{\longrightarrow}\:
	\tfrac{1}{2}\norm{\al-\be}_{-\C}^2.
}

\paragraph{The entropic bias.}
Why bother with the auto-correlation terms $\OT_\epsilon(\al,\al)$
and $\OT_\epsilon(\be,\be)$?
For positive values of $\epsilon$, in general, $\OT_\epsilon(\be,\be) \neq 0$ 
so that minimizing $\OT_\epsilon(\al,\be)$ with respect to $\al$ results in a biased solution:
as evidenced by Figure~\ref{fig-example-intial},
the gradient of $\OT_\epsilon$ drives $\al$ towards a shrunk
measure whose support is \emph{smaller} than that
of the target measure $\be$.
This is most evident as $\epsilon$ tends to infinity: 
$\OT_\epsilon(\al,\be)\rightarrow \iint \C(x,y)\,\d\al(x)\,\d\be(y)$, a quantity that is minimized
if $\al$ is a Dirac distribution located at the median 
(\emph{resp.} the mean) value of $\beta$
if $\C(x,y)=\|x-y\|$ (\emph{resp.} $\|x-y\|^2$).

In the literature, the formula \eqref{eq-sinkhorn-div} has been introduced
more or less empirically to fix the \emph{entropic bias} 
present in the $\OT_\epsilon$ cost:
with a structure that mimicks that of a squared kernel
norm~\eqref{def-mmd}, it was assumed or conjectured that $\S_\epsilon$
would define a positive definite loss function, suitable for
applications in ML.
This paper is all about \emph{proving} that this is indeed what happens.
%



\subsection{Contributions}
The purpose of this paper is to show that the Sinkhorn divergences are convex, smooth, positive definite loss functions that metrize the convergence in law.
Our main result is the theorem below,
that ensures that one can indeed use $\S_\epsilon$ as a \emph{reliable} loss function for ML applications --
whichever value of $\epsilon$ we pick.

\begin{thm}
\label{thm-global-positivity}
  Let $\Xx$ be a compact metric space with a Lipschitz cost function $\C(x,y)$ that induces,
  for $\epsilon>0$, a \emph{positive universal} kernel $k_\epsilon(x,y) \eqdef \exp(-\C(x,y)/\epsilon)$.
  Then, $\S_\epsilon$ defines a symmetric positive definite,
  smooth loss function that is convex in each of its input variables.
  It also metrizes the convergence in law:
  for all probability Radon measures $\al$ and $\be\in\Mm_1^+(\Xx)$,
	\begin{align}
	  0~=~&\Wb_\epsilon(\be,\be)~\leqslant~\Wb_\epsilon(\al,\be),~ \label{eq:positivity_sinkh}\\
	\al=\be~&~\Longleftrightarrow~~\Wb_\epsilon(\al,\be)=0, \label{eq:divergence_definite}\\
	\al_n\rightharpoonup\al~&~\Longleftrightarrow~~\Wb_\epsilon(\al_n,\al) \label{eq:weak-convergence}
\rightarrow 0.
	\end{align}
	Notably, these results also hold for measures
	with bounded support on a Euclidean space
	$\Xx=\R^\D$ endowed with ground cost
	functions $\C(x,y)=\|x-y\|$ or~~ $\C(x,y)=\|x-y\|^2$ --
	which induce Laplacian and Gaussian kernels
	respectively.
\end{thm}

This theorem legitimizes the use of the \emph{unbiased}
Sinkhorn divergences $\S_\epsilon$ instead of $\W_\epsilon$ in model-fitting applications.
Indeed, computing $\Wb_\epsilon$ is roughly as expensive as $\W_\epsilon$
(the computation of the corrective factors
being cheap, as detailed in Section~\ref{sec:practice}) and the ``debiasing''
formula~\eqref{eq-sinkhorn-div} allows us to guarantee
that the unique minimizer of $\al \mapsto \Wb_\epsilon(\al,\be)$
is the target distribution $\be$ (see Figure~\ref{fig-example-intial}).
%
Section~\ref{sec:practice} details how to implement
these divergences efficiently: our algorithms scale up to millions of samples
thanks to freely available GPU routines.
To conclude, we showcase
in Section~\ref{sec:flow} the typical behavior of $\S_\epsilon$
compared with $\OT_\epsilon$ and standard MMD losses.

\section{Proof of Theorem~\ref{thm-global-positivity}}
\label{sec:theory}

We now give the proof of Theorem~\ref{thm-global-positivity}. 
Our argument relies on a new Bregman divergence derived from a weak$^*$ continuous entropy that we call the \emph{Sinkhorn entropy} (see Section~\ref{sec-sink-div}). We believe this (convex) entropy function to be of independent interest. 
Note that all this section is written under the assumptions
of Theorem~\ref{thm-global-positivity};
the proof of some intermediate results can be found in the appendix. 

\begin{figure}[b!]
\centering

\begin{subfigure}[t]{.4\linewidth}
\centering
\includegraphics[width=\textwidth]{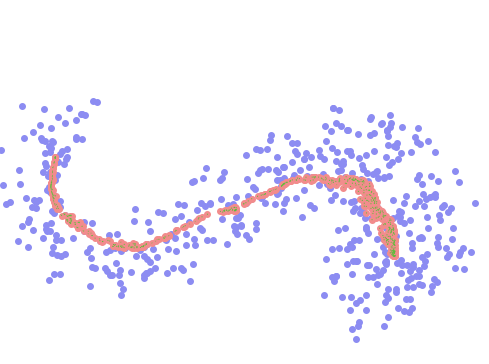}
\caption{$\L=\W_\epsilon$}
\end{subfigure}
\qquad
\begin{subfigure}[t]{.4\linewidth}
\centering
\includegraphics[width=\textwidth]{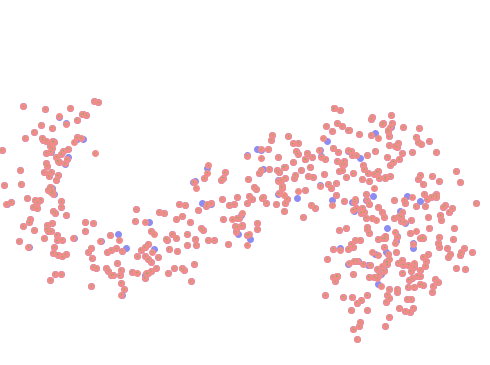}
\caption{$\L=\Wb_\epsilon$}
\end{subfigure}
\caption{\label{fig-example-intial}\textbf{Removing the entropic bias.}
Solution $\al$ (in red) of the fitting problem $\min_\al \L(\al,\be)$ for some $\be$ shown in blue.
Here, $\C(x,y)=\|x-y\|$ on the unit square $\Xx$ in $\R^2$ and
$\epsilon=\texttt{.1}$.
The positions of the red dots were optimized by gradient descent,
starting from a normal Gaussian sample.
}
\end{figure}

\subsection{Properties of the $\OT_\epsilon$ loss}
\label{sec:entrop-reg}

First, let us recall some standard results
of regularized OT theory~\citep{peyre2017computational}.
Thanks to the Fenchel-Rockafellar theorem,
we can rewrite Cuturi's loss~\eqref{eq-primal-ot} as
\begin{align} \label{eq-dual}
\W_\epsilon(\al,\be)
\eqdef 
\max_{(\f, \g)\in \Cc(\Xx)^2}&
\langle\al,\f\rangle
+
\langle\be,\g\rangle \\
-\epsilon 
\langle \al\otimes\be &, 
\exp\big( \tfrac{1}{\epsilon}(\f\oplus \g - \C ) \big) - 1 \rangle,\nonumber
\end{align}
where $\f\oplus\g$ is the tensor sum $(x,y) \in \Xx^2\mapsto \f(x) + \g(y)$.
The primal-dual relationship linking an optimal transport plan
$\pi$ solving~\eqref{eq-primal-ot} to an optimal
dual pair $(\f,\g)$ that solves~\eqref{eq-dual} is
\begin{align}
	\label{eq-primal-dual}
	\pi
	=
	\exp \big( \tfrac{1}{\epsilon} 
	(\f\oplus\g-\C) \big)\cdot(\al\otimes\be).
\end{align}
Crucially, the first order optimality conditions for the dual variables
are equivalent to the primal's marginal constraints
$(\pi_1=\al, \pi_2=\be)$ on~\eqref{eq-primal-dual}.
They read
\begin{align}\label{eq:optim-conditions-dual}
	 \f = \T(\be,\g)\;~ \text{$\al$-a.e.}
	\qandq
	\g = \T(\al,\f)\;~ \text{$\be$-a.e.},
\end{align}
where the ``Sinkhorn mapping'' $\T: \Mm_1^+(\Xx)\times\Cc(\Xx) \rightarrow \Cc(\Xx)$
is defined through
\eql{\label{eq:sinkhorn-op}
	\T : (\al,\f) \mapsto \bigg( y\in\Xx \mapsto 
	\mine{x\sim\al} \big[ 
	 \C(x,y) -\f(x)	
	\big]
	\bigg),
}
with a SoftMin operator of strength $\epsilon$ defined through
\eql{\label{eq:defn-softmin}
	\mine{x \sim \al} \phi(x) 
	\eqdef 
	-\epsilon\log \int_\Xx \exp\big(-\tfrac{1}{\epsilon} \phi(x)\big)\, \d\al(x).
\vspace{-.3cm}
}

\paragraph{Dual potentials.}

The following proposition recalls some important properties of $\OT_\epsilon$ and the associated dual potentials. Its proof can be found in Section~\ref{appendix:dual-pot}.  

\begin{prop}[Properties of $\W_\epsilon$]\label{prop:dual-pot-properties}
	The optimal potentials $(\f,\g)$ exist and are unique $(\al,\be)$-a.e. up to an additive constant, i.e. $\forall K \in\mathbb{R}, \,(\f + K, \g -K)$ is also optimal.
	At optimality, we get
	\eql{\label{eq:strong_duality}
		\W_\epsilon(\al,\be) = \langle\al,\f\rangle+\langle\be,\g\rangle.
	}
\end{prop}

We recall that a function $\F : \Mm_1^+(\Xx) \rightarrow \RR$ is said to be \emph{differentiable} if there exists $\nabla \F(\al) \in \Cc(\Xx)$ such that for any displacement $\xi = \be-\be'$ with $(\be,\be') \in \Mm_1^+(\Xx)^2$,  we have 
\eql{ \label{eq:gradient-Radon}
	\F(\al+t\xi) = \F(\al) + t \dotp{\,\xi}{\nabla \F(\al)} + o(t). 
}
The following proposition, whose proof is detailed in Section~\ref{appendix:prop:differentiability-ot}, shows that the dual potentials are the gradients of $\OT_\epsilon$. 

\begin{prop}\label{prop:differentiability-ot}
	 $\W_\epsilon$ is \emph{weak* continuous} and \emph{differentiable}. Its gradient reads
  	\begin{align}\label{ot-gradient}
    	\nabla \W_\epsilon (\al,\be) = (\f,\g)    
  	\end{align}
	where $(\f,\g)$ satisfies $\f= \T(\be,\g)$ and $\g= \T(\al,\f)$ on the whole domain $\Xx$ and $\T$ is the Sinkhorn mapping~\eqref{eq:sinkhorn-op}.
\end{prop}

Let us stress that even though the solutions of the dual problem~\eqref{eq-dual} are defined $(\al,\be)$-a.e., the gradient~\eqref{ot-gradient} is defined on the whole domain $\Xx$. Fortunately, an optimal dual pair $(\f_0,\g_0)$ defined $(\al,\be)$-a.e. satisfies the optimality condition~\eqref{eq:optim-conditions-dual} and can
be \emph{extended} in a canonical way: 
to compute the ``gradient'' pair $(\f,\g)\in\Cc(\Xx)^2$
associated to a pair of measures $(\al,\be)$,
using $\f= \T(\be,\g_0)$ and $\g= \T(\al,\f_0)$ is enough.

\subsection{Sinkhorn and Haussdorf divergences}
\label{sec-sink-div}

Having recalled some standard properties
of $\OT_\epsilon$, 
let us now state a few \emph{original} facts about 
the corrective, symmetric term 
$-\tfrac{1}{2}\W_\epsilon(\al,\al)$ used in~\eqref{eq-sinkhorn-div}.
We still suppose
that $(\Xx,\d)$ is a compact set endowed with
a symmetric, Lipschitz cost function $\C(x,y)$.
For $\epsilon>0$, the associated \emph{Gibbs kernel} is defined through
 \begin{align}
  k_\epsilon : (x,y)\in\Xx\times\Xx\mapsto\exp\big( -\C(x,y)/\epsilon\big).
\end{align}
Crucially, we now assume that $k_\epsilon$ is a \emph{positive universal} 
kernel on the space of signed Radon measures.

\begin{defn}[Sinkhorn negentropy]\label{def:sinkh-entrop}
Under the assumptions above, we define
the Sinkhorn negentropy of a probability Radon measure
$\al\in\Mm_1^+(\Xx)$ through
\begin{align}\label{form:sinkh-entrop}
\F_\epsilon(\al)\eqdef-\tfrac{1}{2}\W_\epsilon(\al,\al).
\end{align}
\end{defn}

The following proposition is the cornerstone of our approach to prove the positivity of $\S_\epsilon$, providing an alternative expression of $\F_\epsilon$.
Its proof relies on a change of variables
$\mu=\exp(\f/\epsilon)\,\al$
in \eqref{eq-dual} that is detailed in the Section~\ref{appendix:prop:chgt_variable} of the appendix.

\begin{prop}
\label{prop:chgt_variable}
Let $(\Xx,\d)$ be a compact set endowed with a symmetric, Lipschitz cost function $\C(x,y)$ that induces a \emph{positive} kernel $k_\epsilon$.
Then, for $\epsilon > 0$ and $\al\in\Mm_1^+(\Xx)$, one has
\begin{alignat}{8}
\label{eq-sinkh-entropy}
	\tfrac{1}{\epsilon}\F_\epsilon(\al)
	+\tfrac{1}{2}&= 
& &\min_{\mu\in\Mm^+(\Xx)}
&\langle \al, \log \tfrac{\d \al}{\d \mu} \rangle
&+ \tfrac{1}{2} \|\mu\|_{k_\epsilon}^2
&&. 
\end{alignat}
\end{prop}

The following proposition, whose proof can be found in the Section~\ref{appendix:prop:concavity} of the appendix, leverages the alternative expression~\eqref{eq-sinkh-entropy} to ensure the convexity of $\F_\epsilon$.

\begin{prop}
\label{prop:concavity}
Under the same hypotheses as Proposition~\ref{prop:chgt_variable}, $\F_\epsilon$ is a \emph{strictly convex} functional
  on $\Mm_1^+(\Xx)$.
\end{prop}

We now define an auxiliary ``Hausdorff'' divergence
that can be interpreted as an $\OT_\epsilon$ loss with \emph{decoupled} dual potentials.

\begin{defn}[Hausdorff divergence]
Thanks to Proposition~\ref{prop:differentiability-ot},
the Sinkhorn negentropy $\F_\epsilon$
is differentiable in the sense of~\eqref{eq:gradient-Radon}.
For any probability measures $\al,\be \in \Mm_1^+(\Xx)$ and regularization strength $\epsilon > 0$, 
we can thus define 
\begin{align*}
	\H_\epsilon(\al,\be)
	&\eqdef \tfrac{1}{2}\langle\al-\be,
	\nabla \F_\epsilon(\al)-\nabla \F_\epsilon(\be)\rangle~\geqslant~0.
\end{align*}
It is the symmetric Bregman divergence induced by the strictly convex functional
$F_\epsilon$~\citep{bregman1967relaxation}
and is therefore a positive definite quantity.
\end{defn}

\subsection{Proof of the Theorem}

We are now ready to conclude.
First, remark that the dual expression~\eqref{eq-dual} of $\W_\epsilon(\al,\be)$ as a maximization of linear forms ensures that $\W_\epsilon(\al,\be)$ 
is \emph{convex} with respect to $\al$ and with respect to $\be$ 
(but \emph{not} jointly convex if $\epsilon>0$). 
$\Wb_\epsilon$ is thus convex with respect to both inputs $\al$
and $\be$ as a sum of the functions $\W_\epsilon$ and $\F_\epsilon$
-- see Proposition~\ref{prop:concavity}.

Convexity also implies that,
\begin{alignat*}{8}
		\W_\epsilon(\al,\al) &\,+\,& \dotp{\be-\al}{\nabla_2 \W_\epsilon(\al,\al)}
		&~\leqslant~& \W_\epsilon(\al,\be), 
		\\
		\W_\epsilon(\be,\be) &\,+\,& \dotp{\al-\be}{\nabla_1 \W_\epsilon(\be,\be)}
		&~\leqslant~& \W_\epsilon(\al,\be). 
\end{alignat*}
Using \eqref{ot-gradient} to get
$\nabla_2 \W_\epsilon(\al,\al)=-\nabla \F_\epsilon(\al)$,
$\nabla_1 \W_\epsilon(\be,\be)=-\nabla \F_\epsilon(\be)$ and summing the above inequalities, we show that  $\H_\epsilon \leqslant \Wb_\epsilon$,
which implies~\eqref{eq:positivity_sinkh}.

To prove~\eqref{eq:divergence_definite},
note that 
$\Wb_\epsilon(\al,\be)=0\Rightarrow\H_\epsilon(\al,\be)=0$, 
which implies that $\al=\be$ since $\F_\epsilon$ is a \emph{strictly} convex functional.

Finally, we show that $\S_\epsilon$ metrizes the convergence in
law~\eqref{eq:weak-convergence} in the Section~\ref{appendix:metrize-conv-law} of the appendix.


\section{Computational scheme}
\label{sec:practice}

We have shown that Sinkhorn divergences~\eqref{eq-sinkhorn-div}
are positive definite, convex
loss functions on the space of probability measures.
Let us now detail their \emph{implementation}
on modern hardware.

\paragraph{Encoding measures.}

For the sake of simplicity,
we focus on discrete, \emph{sampled} measures
on a Euclidean feature space $\Xx\subset\R^\D$.
Our input measures $\al$ and $\be\in\Mm_1^+(\Xx)$
are represented as sums of weighted Dirac atoms
\begin{align}
\al&=\sum_{i=1}^\N \AL_i\,\delta_{\X_i},
&
\be&=\sum_{j=1}^\M \BE_j\,\delta_{\Y_j}\label{eq:discrete-measures}
\end{align}
and encoded as two pairs $(\AL,\X)$
and $(\BE,\Y)$ of float arrays.
Here, $\AL\in\R^\N_+$ and 
$\BE\in\R^\M_+$
are \emph{non-negative} vectors
of shapes $[\N]$ and $[\M]$
that sum up to~1,
whereas $\X \in (\R^\D)^\N$ and 
$\Y \in (\R^\D)^\M$ are real-valued
tensors of shapes $[\N,\D]$ and $[\M,\D]$
-- if we follow python's convention.

\subsection{The Sinkhorn algorithm(s)}

\paragraph{Working with dual vectors.}
Proposition~\ref{prop:dual-pot-properties} is key to the modern theory
of regularized Optimal Transport:
it allows us to compute the $\OT_\epsilon$ cost --
and thus the Sinkhorn divergence $\S_\epsilon$,
thanks to \eqref{eq-sinkhorn-div} --
using dual variables that have the same \emph{memory footprint}
as the input measures:
solving~\eqref{eq-dual} in our discrete setting, we only need to store
the \emph{sampled values} of the dual potentials $\f$ and $\g$
on the measures' supports.

We can thus work with \emph{dual vectors}
$\B\in\R^\N$ and $\A \in \R^\M$, defined through
$\B_i = \b(\X_i)$ and $\A_j = \a(\Y_j)$,
which encode an \emph{implicit} transport plan $\pi$
from $\al$ to $\be$~\eqref{eq-primal-dual}.
Crucially, the optimality condition
\eqref{eq:optim-conditions-dual}
now reads: \\[.2cm]
$ \forall\,i\in[1,\N],\,  \forall\,j\in[1,\M],$
\begin{align}\hspace*{-.3cm}
\B_i&= -\epsilon\LSE_{k=1}^{\M}\big( \log(\BE_k)+\tfrac{1}{\epsilon}\A_k - \tfrac{1}{\epsilon}\C(\X_i,\Y_k) \big)\label{eq:optimality_lse_discrete_b}\\
\hspace*{-.3cm}
\A_j&= -\epsilon\LSE_{k=1}^{\N}\big( \log(\AL_k)+\tfrac{1}{\epsilon}\B_k - \tfrac{1}{\epsilon}\C(\X_k,\Y_j) \big)\label{eq:optimality_lse_discrete_a} \\
&\text{where}~~~~~\LSE_{k=1}^{\N}(V_k) ~=~ \log \sum_{k=1}^\N \exp(V_k)\label{eq:logsumexp}
\end{align}
denotes a (stabilized) log-sum-exp reduction.

If $(\B,\A)$ is an optimal pair of dual vectors
that satisfies Equations (\ref{eq:optimality_lse_discrete_b}-\ref{eq:optimality_lse_discrete_a}), we deduce from
\eqref{eq:strong_duality} that
\begin{align}\label{eq:strong_duality_discret}
\OT_\epsilon(\AL_i,\X_i,\BE_j,\Y_j)
=
\sum_{i=1}^\N \AL_i \B_i+\sum_{j=1}^\M\BE_j\A_j.
\end{align}
But how can we solve this \emph{coupled} system of equations
given $\AL$, $\X$, $\BE$ and $\Y$ as input data?

\paragraph{The Sinkhorn algorithm.}
One simple answer:
by enforcing \eqref{eq:optimality_lse_discrete_b}
and \eqref{eq:optimality_lse_discrete_a} alternatively,
updating the vectors $\B$ and $\A$ until convergence~\citep{CuturiSinkhorn}.
Starting from null potentials $\B_i=0=\A_j$,
this numerical scheme is nothing but
a block-coordinate ascent on the dual problem~\eqref{eq-dual}.
One step after another, we are enforcing 
null derivatives on the dual cost with respect to
the $\B_i$'s and the $\A_j$'s.

\paragraph{Convergence.}
The ``Sinkhorn loop''
converges quickly towards its unique optimal value:
it enjoys a linear convergence rate~\citep{peyre2017computational}
that can be improved with some heuristics~\citep{thibault2017overrelaxed}.
When computed through the \emph{dual} expression
\eqref{eq:strong_duality_discret},
$\OT_\epsilon$ and its gradients (\ref{eq:gradient_ai}-\ref{eq:gradient_xi})
are \emph{robust}
to small perturbations of the values of $\B$ and $\A$:
monitoring convergence
through the $\text{\upshape{L}}^1$ norm of the updates
on $\B$ and breaking the loop as we reach a set tolerance level
is thus a sensible stopping criterion.
In practice, if $\epsilon$ is large enough -- say, $\epsilon\geqslant\texttt{.05}$
on the unit square with an Earth Mover's cost
$\C(x,y)=\|x-y\|$ -- waiting for 10 or 20 iterations
is more than enough.

\paragraph{Symmetric $\OT_\epsilon$ problems.}
All in all, the baseline Sinkhorn loop provides
an efficient way of solving the discrete problem
$\OT_\epsilon(\al,\be)$ for generic input measures.
But in the specific case of the (symmetric) corrective terms
$\OT_\epsilon(\al,\al)$ and $\OT_\epsilon(\be,\be)$
introduced in~\eqref{eq-sinkhorn-div}, we can do better.

The key here is to remark that if $\al=\be$,
the dual problem~\eqref{eq-dual} becomes
a concave maximization problem that is \emph{symmetric}
with respect to its two variables $\f$ and $\g$.
Hence, there exists a (unique) optimal dual
pair $(\f, \g=\f)$ on the diagonal
which  is characterized in the discrete setting
by the symmetric optimality condition:\\[.2cm]
$\forall i\in[1,\N],$
\begin{align}
\B_i= -\epsilon\LSE_{k=1}^\N\big[
\log(\AL_k) +\tfrac{1}{\epsilon}\B_k
-\tfrac{1}{\epsilon}\C(\X_i,\X_k)
\big].\label{eq:symmetric_update_naive}
\end{align}
Fortunately, given $\AL$ and $\X$,
the optimal vector $\B$ that solves this equation
can be computed by iterating a \emph{well-conditioned} 
fixed-point update:
\begin{align}
\B_i\gets \tfrac{1}{2}\big(
\B_i -\epsilon\LSE_{k=1}^\N\big[
\log(\AL_k) +\tfrac{1}{\epsilon}\B_k
-\tfrac{1}{\epsilon}\C(\X_i,\X_k)
 \big]\,\big).\label{eq:symmetric_update}
\end{align}
This symmetric variant of the Sinkhorn algorithm 
can be shown to converge much
faster than the standard loop applied to a pair
$(\al,\be=\al)$ on the diagonal, and three iterations are
usually enough to compute accurately the optimal dual vector.

\subsection{Computing the Sinkhorn divergence and its gradients}

Given two pairs $(\AL,\X)$ and $(\BE,\Y)$ of float arrays
that encode the probability measures $\al$ and
$\be$~\eqref{eq:discrete-measures},
we can now implement the Sinkhorn divergence $\S_\epsilon(\al,\be)$:\\
The \emph{cross-correlation} dual vectors
$\B\in\R^\N$ and $\A\in\R^\M$ associated to the discrete
problem $\OT_\epsilon(\al,\be)$ can be computed using
the Sinkhorn iterations
(\ref{eq:optimality_lse_discrete_b}-\ref{eq:optimality_lse_discrete_a}).\\
The \emph{autocorrelation} dual vectors
$\P\in\R^\N$ and $\Q\in\R^\M$,
respectively associated to the symmetric problems
$\OT_\epsilon(\al,\al)$ and $\OT_\epsilon(\be,\be)$,
can be computed using the \emph{symmetric}
Sinkhorn update \eqref{eq:symmetric_update}.
\\
The Sinkhorn \emph{loss} can be computed using
\eqref{eq-sinkhorn-div} and \eqref{eq:strong_duality_discret}:
\begin{align*}
\S_\epsilon(\AL_i,\X_i,\BE_j,\Y_j)
=
\sum_{i=1}^\N \AL_i (\B_i-\P_i )+\sum_{j=1}^\M\BE_j(\A_j-\Q_j ).
\end{align*}

\paragraph{What about the gradients?}
In this day and age, we could be tempted to rely
on the \emph{automatic differentiation} engines
provided by modern libraries,
which let us differentiate the result of 
twenty or so Sinkhorn iterations
as a mere composition of elementary operations~\citep{genevay2018learning}.
But beware: this loop has a lot more
\emph{structure} than a generic feed forward network.
Taking advantage of it is key to a x2-x3 gain in performances,
as we now describe.

Crucially, we must remember that the Sinkhorn loop 
is a \emph{fixed point}
iterative solver: at convergence, its solution
satisfies an equation given by the implicit function theorem.
Thanks to~\eqref{ot-gradient}, 
using the very definition of gradients in the space
of probability measures~\eqref{eq:gradient-Radon}
and the intermediate variables in the computation
of $\S_\epsilon(\al,\be)$,
we get that
\begin{align}
\partial_{\AL_i}\SS_\epsilon(\AL_i,\X_i,\BE_j,\Y_j)
~&=~\B_i-\P_i \label{eq:gradient_ai} \\
\text{and}~~~~~~\partial_{\X_i}\SS_\epsilon(\AL_i,\X_i,\BE_j,\Y_j)
~&=~\nabla \phi(\X_i),
\label{eq:gradient_xi}
\end{align}
where $\phi:\Xx\rightarrow \R$ is
equal to $\B_i-\P_i$ on the $\X_i$'s and is defined through
\begin{align*}
\phi(x)
~=~&-\epsilon\,\log\sum_{j=1}^\M\exp\big[\,
\log(\BE_j)+\tfrac{1}{\epsilon}\A_j-\tfrac{1}{\epsilon}\C(x,\Y_j)
\,\big]\\
&+\epsilon\,\log\sum_{i=1}^\N\exp\big[\,
\log(\AL_i)+\tfrac{1}{\epsilon}\P_i-\tfrac{1}{\epsilon}\C(x,\X_i)
\,\big].
\end{align*}

\begin{figure*}[t!]
    \centering
    \makebox[\textwidth][c]{
        \mbox{
\pgfplotsset{grid style={gray!25}}
\begin{minipage}{.5\textwidth}
    \centering
    \resizebox {.85\columnwidth} {!} {
    \begin{tikzpicture}
      \begin{axis}[scale=.9,grid=both,ymin=1e-4, ymax=300, xmin=1e2, xmax=1e6, 
                   title={Computing an Energy Distance + gradient between samples of size $\N=\M$},
                   xlabel={Number of points $\N$}, 
                   ylabel=Time (sec), 
                   xmode=log, ymode=log,
                   legend pos = south east,
                   x post scale=1.7,
                   grid=major,
                   legend cell align={left}]
        \addplot[green!50!black, thick,mark=*] table[x=Npoints,y=pytorch_cpu]  {sections/images/benchmark_energy_distance.csv} node[right]{\footnotesize \oma{\textbf{~out of mem}}};
        \addplot[blue!50, thick,mark=*] table[x=Npoints,y=pytorch_gpu]  {sections/images/benchmark_energy_distance.csv} node[right]{\footnotesize \omb{\textbf{~out of mem}}};
        \addplot[red!80, thick,mark=*] table[x=Npoints,y=GPU_1D]  {sections/images/benchmark_energy_distance.csv};
        \addlegendentry{\texttt{PyTorch} on CPU}
        \addlegendentry{\texttt{PyTorch} on GPU}
        \addlegendentry{\texttt{PyTorch + KeOps}}
      \end{axis}
    \end{tikzpicture}
    }
\end{minipage}
\begin{minipage}{.5\textwidth}
    \centering
    \resizebox {.85\columnwidth} {!} {
    \begin{tikzpicture}
      \begin{axis}[scale=.9,grid=both,ymin=1e-4, ymax=300, xmin=1e2, xmax=1e6, 
                   title={Computing $\log\sum_j\exp\|x_i-y_j\|$ with samples of size $\N=\M$},
                   xlabel={Number of points $\N$}, 
                   ylabel=Time (sec), 
                   xmode=log, ymode=log,
                   legend pos = south east,
                  x post scale=1.7,
                   grid=major,
                   legend cell align={left}]
        \addplot[green!50!black, thick,mark=*] table[x=Npoints,y=pytorch_cpu]  {sections/images/benchmark_LogSumExp.csv} node[right]{\footnotesize \oma{\textbf{~out of mem}}};
        \addplot[blue!50, thick,mark=*] table[x=Npoints,y=pytorch_gpu]  {sections/images/benchmark_LogSumExp.csv} node[right]{\footnotesize \omb{\textbf{~out of mem}}};
        \addplot[red!80, thick,mark=*] table[x=Npoints,y=GPU_1D]  {sections/images/benchmark_LogSumExp.csv};
        \addlegendentry{\texttt{PyTorch} on CPU}
        \addlegendentry{\texttt{PyTorch} on GPU}
       \addlegendentry{\texttt{PyTorch + KeOps}}
      \end{axis}
    \end{tikzpicture}
    }
\end{minipage}
        }
    }
\vspace{-.4cm}
\caption{\textbf{The KeOps library allows us to break the memory bottleneck.}
    Using CUDA routines that sum kernel values
    without storing them in memory,
    we can outperform baseline, tensorized,
    implementations of the energy distance.
    Experiments performed on $\Xx=\R^3$
    with a cheap laptop's GPU (GTX 960M).\label{fig:bench-keops}}
    
    \vspace{.2cm}

    \makebox[\textwidth][c]{
        \mbox{
\pgfplotsset{grid style={gray!25}}
\begin{minipage}{.5\textwidth}
    \centering
    \resizebox {.85\columnwidth} {!} {
    \begin{tikzpicture}
      \begin{axis}[scale=.9,grid=both,ymin=1e-4, ymax=300, xmin=1e2, xmax=1e6, 
                   title={On a cheap laptop's GPU (GTX960M)},
                   xlabel={Number of points $\N$}, 
                   ylabel=Time (sec), 
                   xmode=log, ymode=log,
                   legend pos = north west,
                   x post scale=1.7,
                   legend cell align={left},
                   grid=major,
                   ]
        \addplot[green!50!black, thick,mark=*] table[x=Npoints,y=sinkhorn_nocv]  {sections/images/benchmark_fidelities_GTX960M.csv};
        \addplot[blue!50, thick,mark=*] table[x=Npoints,y=sinkhorn]  {sections/images/benchmark_fidelities_GTX960M.csv};
        \addplot[red!80, thick,mark=*] table[x=Npoints,y=energy_distance]  {sections/images/benchmark_fidelities_GTX960M.csv};
        \addlegendentry{Sinkhorn divergence (naive)}
        \addlegendentry{Sinkhorn divergence}
        \addlegendentry{Energy Distance}
      \end{axis}
    \end{tikzpicture}
    }
\end{minipage}
\begin{minipage}{.5\textwidth}
    \centering
    \resizebox {.85\columnwidth} {!} {
    \begin{tikzpicture}
      \begin{axis}[scale=.9,grid=both,ymin=1e-4, ymax=300, xmin=1e2, xmax=1e6, 
                   title={On a high end GPU (P100)},
                   xlabel={Number of points $\N$}, 
                   ylabel=Time (sec), 
                   xmode=log, ymode=log,
                   legend pos = north west,
                   x post scale=1.7,
                   legend cell align={left},
                   grid=major,
                   ]
        \addplot[green!50!black, thick,mark=*] table[x=Npoints,y=sinkhorn_nocv]  {sections/images/benchmark_fidelities_P100.csv};
        \addplot[blue!50, thick,mark=*] table[x=Npoints,y=sinkhorn]  {sections/images/benchmark_fidelities_P100.csv};
        \addplot[red!80, thick,mark=*] table[x=Npoints,y=energy_distance]  {sections/images/benchmark_fidelities_P100.csv};
        \addlegendentry{Sinkhorn divergence (naive)}
        \addlegendentry{Sinkhorn divergence }
        \addlegendentry{Energy Distance}
      \end{axis}
    \end{tikzpicture}
    }
\end{minipage}
        }
    }
\vspace{-.4cm}
\caption{\textbf{Sinkhorn divergences scale up to finely sampled distributions.}\label{fig:bench_divergences}
    As a rule of thumb, Sinkhorn divergences
    take 20-50 times as long to compute as a baseline MMD
    -- even though the explicit gradient formula
    (\ref{eq:gradient_ai}-\ref{eq:gradient_xi})
    lets us win a factor 2-3
    compared with a \emph{naive} autograd implementation.
    For the sake of this benchmark,
    we ran the Sinkhorn and symmetric Sinkhorn loops
    with fixed numbers of iterations: 20 and 3
    respectively, which is more than enough for
    measures on the unit (hyper)cube if $\epsilon \geqslant \texttt{.05}$
    -- here, we work in $\R^3$.}    
    
\end{figure*}

\paragraph{Graph surgery with PyTorch.}
Assuming convergence in the Sinkhorn loops, 
it is thus possible to compute
the gradients of $\S_\epsilon$ 
\emph{without having to backprop} through
the twenty or so iterations of the Sinkhorn algorithm:
we only have to differentiate the expression above
with respect to $x$.
But does it mean that we should differentiate
$\C$ or the log-sum-exp operation by hand?
Fortunately, no!

Modern libraries such as PyTorch~\citep{pytorch}
are flexible enough to let us ``hack'' the
naive \texttt{autograd} algorithm, and act
as though the optimal dual vectors
$\B_i$, $\P_i$, $\A_j$ and $\Q_j$
did not depend on the input variables of $\S_\epsilon$.
As documented in our reference code,
\begin{center}
\texttt{github.com/jeanfeydy/global-divergences},
\end{center}
an appropriate use of the \texttt{.detach()}
method in PyTorch is enough to get the best
of both worlds:
an \emph{automatic} differentiation engine
that computes our gradients using the formula \emph{at convergence}
instead of the baseline backpropagation algorithm.
All in all, as evidenced by the benchmarks
provided Figure~\ref{fig:bench_divergences},
this trick allows us to divide by a factor 2-3
the time needed to compute a Sinkhorn divergence and its gradient
with respect to the $\X_i$'s.

\subsection{Scaling up to large datasets}
The Sinkhorn iterations rely on a single non-trivial operation:
the log-sum-exp reduction~\eqref{eq:logsumexp}.
In the ML literature, this \emph{SoftMax} operator
is often understood as a row- or column-wise
reduction that acts on $[\N,\M]$ matrices.
But as we strive to implement the update rules
(\ref{eq:optimality_lse_discrete_b}-\ref{eq:optimality_lse_discrete_a})
and \eqref{eq:symmetric_update} on the GPU,
we can go further.

\paragraph{Batch computation.}
First, if the number of samples $\N$ and $\M$
in both measures is small enough, we can optimize
the GPU usage by computing Sinkhorn divergences \emph{by batches}
of size $\Batch$.
In practice, this can be achieved by encoding the cost function $\C$ as a 3D tensor
of size $[\Batch,\N,\M]$ made up of stacked matrices $(\C(\X_i,\Y_j))_{i,j}$,
while $\B$ and $\A$ become $[\Batch,\N]$ and $[\Batch,\M]$
tensors, respectively.
Thanks to the broadcasting syntax supported by modern
libraries, we can then seamlessly compute, in parallel,
loss values $\S_\epsilon(\al_k,\be_k)$ for
$k$ in $[1,\Batch]$.

\paragraph{The KeOps library.}
Unfortunately though, tensor-centric methods such as the one presented above
cannot scale to measures sampled with large numbers $\N$ and $\M$
of Dirac atoms: as these numbers exceed 10,000, huge $[\N,\M]$ matrices stop fitting into GPU memories.
To alleviate this problem, we leveraged the KeOps library~\citep{keops}
that provides \emph{online} map-reduce routines on the GPU with full 
PyTorch integration.
Performing online log-sum-exp reductions
with a \emph{running maximum},
the KeOps primitives allow us to compute Sinkhorn
divergences with a \emph{linear} memory footprint.
As evidenced by the benchmarks of Figures~\ref{fig:bench-keops}-\ref{fig:bench_divergences},
computing the gradient of a Sinkhorn loss with 100,000
samples per measure is then a matter of seconds.


\section{Numerical illustration}
\label{sec:flow}

In the previous sections,
we have provided theoretical guarantees
on top of a comprehensive implementation guide for
the family of \emph{Sinkhorn divergences} $\S_\epsilon$.
Let us now describe the \emph{geometry} induced by 
these new loss functions on the space of probability measures.

\paragraph{Gradient flows.}
To compare MMD losses $\L_k$ with Cuturi's original cost $\OT_\epsilon$
and the de-biased Sinkhorn divergence $\S_\epsilon$, 
a simple yet relevant experiment is to let
a \emph{model} distribution $\al(t)$ flow with time $t$ along
the ``Wasserstein-2'' gradient flow of a loss
functional $\al\mapsto \L(\al,\be)$
that drives it towards a target distribution
$\be$~\citep{santambrogio2015optimal}.
This corresponds to the ``non-parametric'' version of the data fitting problem evoked in Section~\ref{sec-intro}, where the parameter $\th$ is nothing but the vector of positions $\X$ that encodes the support of a measure
$\al = \tfrac{1}{\N}\sum_{i=1}^N \delta_{\X_i}$.
Understood as a ``model free'' idealization of fitting problems in machine learning,
this experiment allows us to grasp the typical behavior of the loss function
as we discover the deformations of the support that it favors.

\newcommand{\sidecapY}[1]{ \begin{sideways}\parbox{.14\linewidth}{\centering #1}\end{sideways} }
\newcommand{\FigGF}[1]{\includegraphics[width=.17\linewidth]{sections/images/grad-flow-1d/#1}}
\newcommand{\FigRowGF}[1]{\FigGF{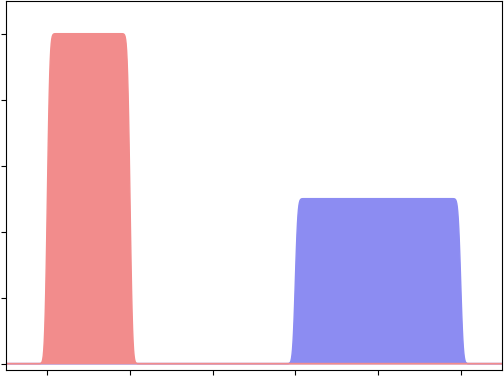} & \FigGF{#1-1} & \FigGF{#1-2} & \FigGF{#1-3} & \FigGF{#1-4}}
\begin{figure}[H]
	\centering
	\begin{tabular}{@{}c@{}c@{\hspace{1mm}}c@{\hspace{1mm}}c@{\hspace{1mm}}c@{\hspace{1mm}}c@{}}
	 \sidecapY{ \tiny $\norm{{\color{red}\al}-{\color{blue}\be}}_{k}^2$
	 ``$\epsilon=\infty$'' }  & \FigRowGF{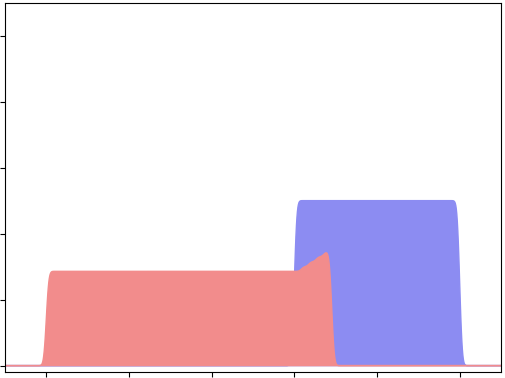} \\
	 \sidecapY{ \tiny $\S_\epsilon({\color{red}\al},{\color{blue}\be})$
	 $\epsilon=\texttt{.10}$ }  & \FigRowGF{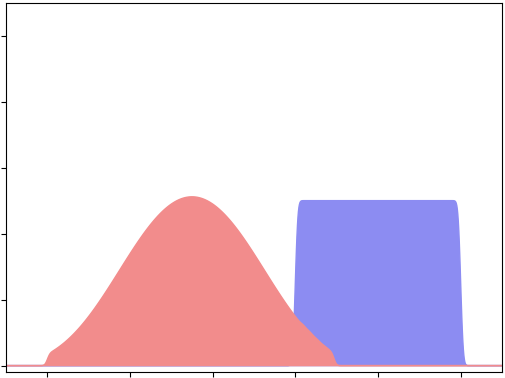} \\
	 \sidecapY{ \tiny $\S_\epsilon({\color{red}\al},{\color{blue}\be})$
	 $\epsilon=\texttt{.01}$ }  & \FigRowGF{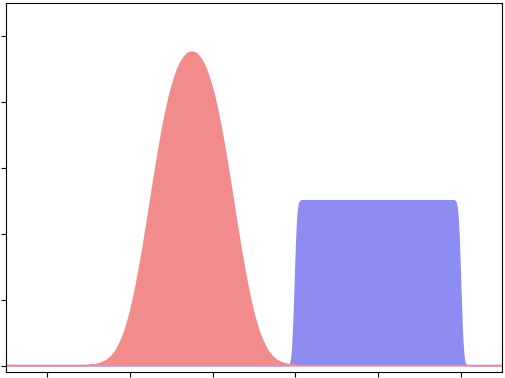} \\
	 \sidecapY{ \tiny $\OT_\epsilon({\color{red}\al},{\color{blue}\be})$
	 $\epsilon=\texttt{.10}$ }   & \FigRowGF{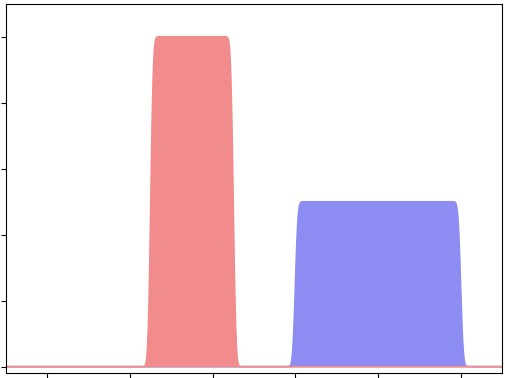} \\
	 \sidecapY{ \tiny  $\OT_0({\color{red}\al},{\color{blue}\be})$
	 $\epsilon=\texttt{0.}$ }  & \FigRowGF{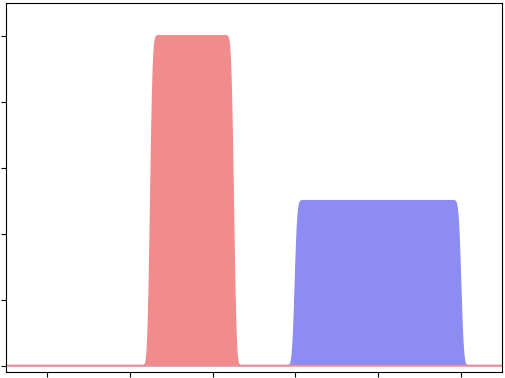} \\
	 & $t=\texttt{0.}$ & $t=\texttt{.25}$ & $t=\texttt{.50}$ & $t=\texttt{1.00}$ & $t=\texttt{5.00}$  
	\end{tabular}
	\caption{Gradient flows for 1-D measures sampled with $\N=\M=5000$ points
	-- we display $\al(t)$ (in red) and $\be$ (in blue) through kernel density estimations on the segment $[0,1]$.
	The legend on the left indicates the function that is minimized
	with respect to $\al$. 
	Here $k(x,y)=-\norm{x-y}$, $\C(x,y)=\norm{x-y}$ and $\epsilon=\texttt{.10}$ on the second and fourth lines,
	$\epsilon=\texttt{.01}$ on the third.
	\\
	In 1D, the optimal transport problem can be solved using
	a sort algorithm: for the sake of comparison,
	we can thus display the ``true'' dynamics of the
	Earth Mover's Distance in the fifth line.
	\label{fig:flow1d}\vspace{-\baselineskip}}
\end{figure}

\renewcommand{\FigGF}[1]{\includegraphics[width=.18\linewidth]{sections/images/grad-flow-2d/#1}}
\renewcommand{\FigRowGF}[1]{\FigGF{#1-0} & \FigGF{#1-1} & \FigGF{#1-2} & \FigGF{#1-3} & \FigGF{#1-4}}

\begin{figure*}[t]
	\centering
	\begin{tabular}{@{}c@{}c@{\hspace{1mm}}c@{\hspace{1mm}}c@{\hspace{1mm}}c@{\hspace{1mm}}c@{}}
	 \sidecapY{  $\norm{{\color{red}\al}-{\color{blue}\be}}_{k}^2$
	 ``$\epsilon=+\infty$''}  & \FigRowGF{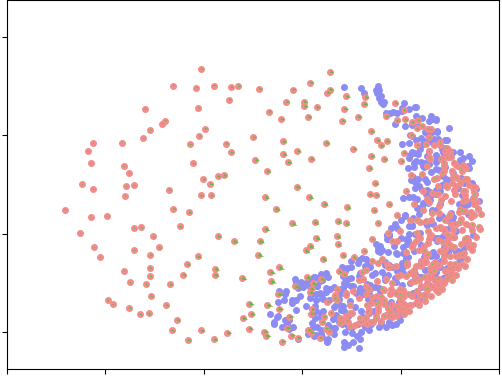} \\
	 \sidecapY{  $\S_\epsilon({\color{red}\al},{\color{blue}\be})$
	 $\epsilon=\texttt{.10}$ } & \FigRowGF{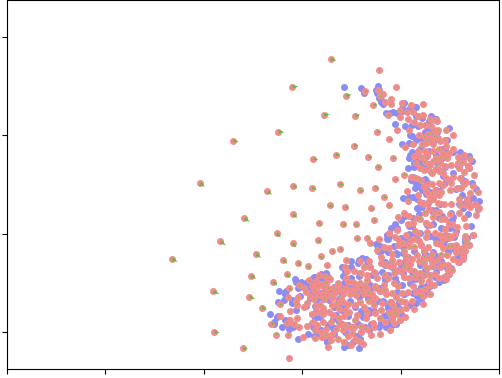} \\
	 \sidecapY{  $\S_\epsilon({\color{red}\al},{\color{blue}\be})$
	 $\epsilon=\texttt{.01}$} & \FigRowGF{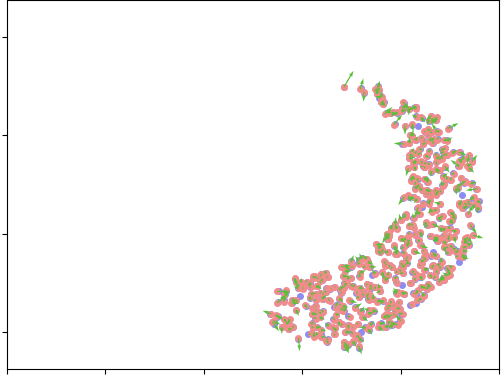} \\
	 \sidecapY{ $\OT_\epsilon({\color{red}\al},{\color{blue}\be})$
	 $\epsilon=\texttt{.10}$ }  & \FigRowGF{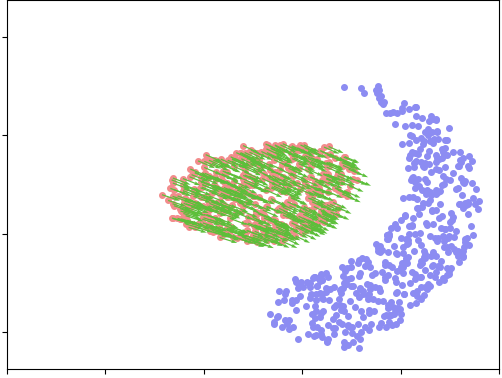} \\
	 & $t=\texttt{0.}$ & $t=\texttt{.25}$ & $t=\texttt{.50}$ & $t=\texttt{1.00}$ & $t=\texttt{5.00}$  
	\end{tabular}
	\caption{\textbf{Gradient flows for 2-D measures.} The setting is the same as in Figure~\ref{fig:flow1d}, but the measures $\al(t)$ (in red) and $\be$ (in blue) are now directly displayed as point clouds of $\N=\M=500$ points. The evolution of the support of $\al(t)$ is thus made apparent,
	and we display as a green vector field the descent direction $-\nabla_{\X_i}\L(\al,\be)$.
	\label{fig:flow2d}}
	\vspace{-\baselineskip}
\end{figure*}

In Figures~\ref{fig:flow1d} and~\ref{fig:flow2d},
$\beta=\tfrac{1}{\M}\sum_{j=1}^\M\delta_{\Y_j}$ is a fixed
target measure while 
$\al = \tfrac{1}{\N}\sum_{i=1}^N \delta_{\X_i(t)}$
is parameterized by a time-varying point cloud
$\X(t)=(\X_i(t))_{i=1}^\N \in (\RR^\D)^\N$ 
in dimension~$\D=1$ or $2$. 
Starting from a set initial condition at time $t=0$,
we simply integrate the ODE
\eq{
	\dot \X(t) = -\N\,\nabla_{\X} \big[ \,\L\big(\tfrac{1}{\N}{\scriptstyle \sum_{i=1}^N }\delta_{\X}, \be \big) \,\big] (\X_i(t))
}
with a Euler scheme and display the
evolution of $\al(t)$ up to time $t=5$.

\paragraph{Interpretation.}
In both figures, the fourth line highlights the 
entropic bias that is present in the $\OT_\epsilon$ loss:
$\al(t)$ is driven towards a minimizer that is a ``shrunk'' 
version of $\be$. 
As showed in Theorem~\ref{thm-global-positivity},
the de-biased loss $\S_\epsilon$ does not suffer from this issue: just like MMD norms, it can be used as a \emph{reliable},
positive-definite divergence.  

Going further, 
the dynamics induced by the Sinkhorn divergence
interpolates between that of an MMD ($\epsilon=+\infty$)
and Optimal Transport ($\epsilon=0$), 
as shown in~\eqref{eq:sinkhorn-interpolation}.
Here, $\C(x,y)=\|x-y\|$ and we can indeed remark that 
the second and third lines bridge the gap
between the flow of the energy distance $\L_{-\|\cdot\|}$
(in the first line) and that
of the Earth Mover's cost $\OT_0$
which moves particles according
to an optimal transport plan.

Please note that in both experiments,
the gradient of the energy distance with respect to the
$\X_i$'s vanishes at the extreme points of $\al$'s support.
Crucially, for small enough values of $\epsilon$,
$\S_\epsilon$ recovers the translation-aware
geometry of OT and we observe a \emph{clean} convergence of $\al(t)$ to $\be$
as no sample lags behind.

\section{Conclusion}


Recently introduced in the ML literature,
the Sinkhorn divergences were designed to
interpolate between MMD and OT.
We have now shown that they also come with
a bunch of desirable properties: 
positivity, convexity, metrization of the convergence in law
and scalability to large datasets.

To the best of our knowledge, it is the first time
that a loss derived from the theory of entropic
Optimal Transport is shown to stand on such a firm ground.
As the foundations of this theory are progressively 
being settled, we now hope that researchers will be free to
focus on one of the major open problems in the field:
the interaction of \emph{geometric} loss functions
with concrete machine learning models.

\bibliographystyle{apalike}
\bibliography{biblio}

\begin{thebibliography}{}

\bibitem[Arjovsky et~al., 2017]{WassersteinGAN}
Arjovsky, M., Chintala, S., and Bottou, L. (2017).
\newblock Wasserstein {GAN}.
\newblock {\em arXiv preprint arXiv:1701.07875}.

\bibitem[Bassetti et~al., 2006]{bassetti2006minimum}
Bassetti, F., Bodini, A., and Regazzini, E. (2006).
\newblock On minimum {Kantorovich} distance estimators.
\newblock {\em Statistics \& probability letters}, 76(12):1298--1302.

\bibitem[Bonneel et~al., 2016]{bonneel2016wasserstein}
Bonneel, N., Peyr{\'e}, G., and Cuturi, M. (2016).
\newblock Wasserstein barycentric coordinates: Histogram regression using
  optimal transport.
\newblock {\em ACM Transactions on Graphics}, 35(4).

\bibitem[Bregman, 1967]{bregman1967relaxation}
Bregman, L.~M. (1967).
\newblock The relaxation method of finding the common point of convex sets and
  its application to the solution of problems in convex programming.
\newblock {\em USSR computational mathematics and mathematical physics},
  7(3):200--217.

\bibitem[Charlier et~al., 2018]{keops}
Charlier, B., Feydy, J., and Glaun\`es, J. (2018).
\newblock Kernel operations on the gpu, with autodiff, without memory
  overflows.
\newblock \url{http://www.kernel-operations.io}.
\newblock Accessed: 2018-10-04.

\bibitem[Cuturi, 2013]{CuturiSinkhorn}
Cuturi, M. (2013).
\newblock Sinkhorn distances: Lightspeed computation of optimal transport.
\newblock In {\em Adv. in Neural Information Processing Systems}, pages
  2292--2300.

\bibitem[Dziugaite et~al., 2015]{MMD-GAN}
Dziugaite, G.~K., Roy, D.~M., and Ghahramani, Z. (2015).
\newblock Training generative neural networks via maximum mean discrepancy
  optimization.
\newblock In {\em Proceedings of the Thirty-First Conference on Uncertainty in
  Artificial Intelligence}, pages 258--267.

\bibitem[Franklin and Lorenz, 1989]{franklin1989scaling}
Franklin, J. and Lorenz, J. (1989).
\newblock On the scaling of multidimensional matrices.
\newblock {\em Linear Algebra and its applications}, 114:717--735.

\bibitem[Frogner et~al., 2015]{2015-Frogner}
Frogner, C., Zhang, C., Mobahi, H., Araya, M., and Poggio, T.~A. (2015).
\newblock Learning with a {W}asserstein loss.
\newblock In {\em Advances in Neural Information Processing Systems}, pages
  2053--2061.

\bibitem[Galichon and Salani{\'e}, 2010]{galichon2010matching}
Galichon, A. and Salani{\'e}, B. (2010).
\newblock Matching with trade-offs: Revealed preferences over competing
  characteristics.
\newblock {\em Preprint hal-00473173}.

\bibitem[Genevay et~al., 2018]{genevay2018learning}
Genevay, A., Peyr{\'e}, G., and Cuturi, M. (2018).
\newblock Learning generative models with sinkhorn divergences.
\newblock In {\em International Conference on Artificial Intelligence and
  Statistics}, pages 1608--1617.

\bibitem[Glaunes et~al., 2004]{glaunes2004diffeomorphic}
Glaunes, J., Trouv{\'e}, A., and Younes, L. (2004).
\newblock Diffeomorphic matching of distributions: A new approach for
  unlabelled point-sets and sub-manifolds matching.
\newblock In {\em Computer Vision and Pattern Recognition, 2004. CVPR 2004.
  Proceedings of the 2004 IEEE Computer Society Conference on}, volume~2, pages
  II--II. Ieee.

\bibitem[Goodfellow et~al., 2014]{goodfellow2014generative}
Goodfellow, I., Pouget-Abadie, J., Mirza, M., Xu, B., Warde-Farley, D., Ozair,
  S., Courville, A., and Bengio, Y. (2014).
\newblock Generative adversarial nets.
\newblock In {\em Advances in neural information processing systems}, pages
  2672--2680.

\bibitem[Gretton et~al., 2007]{gretton2007kernel}
Gretton, A., Borgwardt, K.~M., Rasch, M., Sch{\"o}lkopf, B., and Smola, A.~J.
  (2007).
\newblock A kernel method for the two-sample-problem.
\newblock In {\em Advances in neural information processing systems}, pages
  513--520.

\bibitem[Kaltenmark et~al., 2017]{varifold}
Kaltenmark, I., Charlier, B., and Charon, N. (2017).
\newblock A general framework for curve and surface comparison and registration
  with oriented varifolds.
\newblock In {\em Computer Vision and Pattern Recognition (CVPR)}.

\bibitem[Kantorovich, 1942]{Kantorovich42}
Kantorovich, L. (1942).
\newblock On the transfer of masses (in {R}ussian).
\newblock {\em Doklady Akademii Nauk}, 37(2):227--229.

\bibitem[L{\'e}onard, 2013]{leonard2013survey}
L{\'e}onard, C. (2013).
\newblock A survey of the {Schr{\"o}dinger} problem and some of its connections
  with optimal transport.
\newblock {\em arXiv preprint arXiv:1308.0215}.

\bibitem[Li et~al., 2015]{li2015generative}
Li, Y., Swersky, K., and Zemel, R. (2015).
\newblock Generative moment matching networks.
\newblock In {\em Proceedings of the 32nd International Conference on Machine
  Learning (ICML-15)}, pages 1718--1727.

\bibitem[Micchelli et~al., 2006]{micchelli2006universal}
Micchelli, C.~A., Xu, Y., and Zhang, H. (2006).
\newblock Universal kernels.
\newblock {\em Journal of Machine Learning Research}, 7(Dec):2651--2667.

\bibitem[Montavon et~al., 2016]{montavon2016wasserstein}
Montavon, G., M{\"u}ller, K.-R., and Cuturi, M. (2016).
\newblock Wasserstein training of restricted boltzmann machines.
\newblock In {\em Advances in Neural Information Processing Systems}, pages
  3718--3726.

\bibitem[Paszke et~al., 2017]{pytorch}
Paszke, A., Gross, S., Chintala, S., Chanan, G., Yang, E., DeVito, Z., Lin, Z.,
  Desmaison, A., Antiga, L., and Lerer, A. (2017).
\newblock Automatic differentiation in pytorch.

\bibitem[Peyr{\'e} and Cuturi, 2017]{peyre2017computational}
Peyr{\'e}, G. and Cuturi, M. (2017).
\newblock Computational optimal transport.
\newblock {\em arXiv:1610.06519}.

\bibitem[Ramdas et~al., 2017]{RamdasSinkhAsymptotics}
Ramdas, A., Trillos, N.~G., and Cuturi, M. (2017).
\newblock On wasserstein two-sample testing and related families of
  nonparametric tests.
\newblock {\em Entropy}, 19(2).

\bibitem[Rubner et~al., 2000]{rubner-2000}
Rubner, Y., Tomasi, C., and Guibas, L.~J. (2000).
\newblock The earth mover's distance as a metric for image retrieval.
\newblock {\em International Journal of Computer Vision}, 40(2):99--121.

\bibitem[Salimans et~al., 2018]{salimans2018improving}
Salimans, T., Zhang, H., Radford, A., and Metaxas, D. (2018).
\newblock Improving {GANs} using optimal transport.
\newblock {\em arXiv preprint arXiv:1803.05573}.

\bibitem[Sanjabi et~al., 2018]{SanjabiSinkhGAN}
Sanjabi, M., Ba, J., Razaviyayn, M., and Lee, J.~D. (2018).
\newblock On the convergence and robustness of training {GANs} with regularized
  optimal transport.
\newblock {\em arXiv preprint arXiv:1802.08249}.

\bibitem[Santambrogio, 2015]{santambrogio2015optimal}
Santambrogio, F. (2015).
\newblock {\em Optimal Transport for applied mathematicians}, volume~87 of {\em
  Progress in Nonlinear Differential Equations and their applications}.
\newblock Springer.

\bibitem[Sz{\'e}kely and Rizzo, 2004]{SR05}
Sz{\'e}kely, G.~J. and Rizzo, M.~L. (2004).
\newblock Hierarchical clustering via joint between-within distances: Extending
  ward’s minimum variance method.
\newblock {\em J. Classification}, 22:151--183.

\bibitem[Thibault et~al., 2017]{thibault2017overrelaxed}
Thibault, A., Chizat, L., Dossal, C., and Papadakis, N. (2017).
\newblock Overrelaxed sinkhorn-knopp algorithm for regularized optimal
  transport.
\newblock {\em arXiv preprint arXiv:1711.01851}.

\bibitem[Vaillant and Glaun{\`e}s, 2005]{vaillant2005surface}
Vaillant, M. and Glaun{\`e}s, J. (2005).
\newblock Surface matching via currents.
\newblock In {\em Biennial International Conference on Information Processing
  in Medical Imaging}, pages 381--392. Springer.

\end{thebibliography}

\newpage
\appendix

~\newpage

\section{Standard results}

Before detailing our proofs,
we first recall some well-known results regarding the
Kullback-Leibler divergence and the SoftMin operator defined 
in~\eqref{eq:defn-softmin}.

\subsection{The Kullback-Leibler divergence}

\paragraph{First properties.}
For any pair of Radon measures $\al,\be\in\Mm^+(\Xx)$ on the compact
metric set $(\Xx,\d)$, the Kullback-Leibler divergence is defined through
\begin{align*}
\KL(\al,\be)~\eqdef~
\begin{cases}
\dotp{\al}{\log\tfrac{\d\al}{\d\be} - 1} + \dotp{\be}{1} &\text{if $\al \ll \be$}\\
~~~~+\infty &\text{otherwise.}
\end{cases}
\end{align*}

It can be rewritten as an $f$-divergence
associated to 
\begin{align*}
\psi : x \in\R_{\geqslant 0}\mapsto x\,\log(x) - x + 1 \,\in \R_{\geqslant 0},
\end{align*}
with $0\cdot\log(0)=0$, as 
\begin{align}
\KL(\al,\be)~=~
\begin{cases}
\dotp{\be}{\psi(\tfrac{\d\al}{\d\be})} &\text{if $\al \ll \be$}\\
~~~~+\infty &\text{otherwise.}
\end{cases}\label{eq:KL_fdivergence}
\end{align}
Since $\psi$ is a strictly convex function
with a unique global minimum at $\psi(1)=0$,
we thus get that $\KL(\al,\be)\geqslant 0$
with equality iff. $\al=\be$.

\paragraph{Dual formulation.}
The convex conjugate of $\psi$
is defined for $u\in\R$ by
\begin{align}
\psi^*(u)~&\eqdef~ \sup_{x>0}\,(x u - \psi(x)) \nonumber \\
&=~~~~~ e^u - 1, \nonumber \\
\text{and we have}~~~~&\psi(x) + \psi^*(u)~\geqslant~ xu \label{eq:fenchel_psi}
\end{align}
for all $(x,u)\in \R_{\geqslant 0}\times \R$, with equality
if $x>0$ and $u=\log(x)$.
This allows us to rewrite the Kullback-Leibler
divergence as the solution of a dual concave problem:

\begin{prop}[Dual formulation of $\KL$]
Under the assumptions above,
\begin{align}
\KL(\al,\be)~=~\sup_{h\in\Ff_b(\Xx,\R)}~~ \dotp{\al}{h}-\dotp{\be}{e^h-1} \label{eq:KL_dual_mesurable} 
\end{align}
where $\Ff_b(\Xx,\R)$ is the space of bounded measurable functions from $\Xx$ to $\R$.
\end{prop}
\begin{proof}
\emph{Lower bound on the sup.} If $\al$ is not absolutely
continuous with respect to $\be$, there exists a Borel
set $A$ such that $\al(A) > 0$ and $\be(A)=0$.
Consequently, for $h = \lambda\,\textbf{1}_A$,
\begin{align*}
\dotp{\al}{h}-\dotp{\be}{e^h-1}~=~\lambda\,\al(A)~\xrightarrow{\lambda\rightarrow+\infty}+\infty.
\end{align*}
Otherwise, if $\al\ll\be$, we define
$h_* = \log \tfrac{\d\al}{\d\be}$ and see that
\begin{align*}
\dotp{\al}{h_*}-\dotp{\be}{e^{h_*}-1}~=~
\KL(\al,\be).
\end{align*}
If $h_n = \log (\tfrac{\d\al}{\d\be})\,\textbf{1}_{1/n\leqslant \d\al/\d\be\leqslant n}\in\Ff_b(\Xx,\R)$, the monotone and dominated convergence theorems
then allow us to show that
\begin{align*}
\dotp{\al}{h_n} - \dotp{\be}{e^{h_n}-1} \xrightarrow{n\rightarrow+\infty} 
\KL(\al,\be).
\end{align*}

\emph{Upper bound on the sup.}
If $h\in\Ff_b(\Xx,\R)$ and $\al \ll \be$, combining \eqref{eq:KL_fdivergence}
and \eqref{eq:fenchel_psi}
allow us to show that
\begin{align*}
&\KL(\al,\be) - \dotp{\al}{h} + \dotp{\be}{e^h-1} \\
~=~& \dotp{\be}{ \psi(\tfrac{\d\al}{\d\be}) + \psi^*(h) - h \tfrac{\d\al}{\d\be} }
~\geqslant~0.
\end{align*}
The optimal value of $\dotp{\al}{h}-\dotp{\be}{e^h-1}$
is bounded above and below by $\KL(\al,\be)$:
we get~\eqref{eq:KL_dual_mesurable}.
\end{proof}

Since $\dotp{\al}{h}-\dotp{\be}{e^h-1}$ is a
convex function of $(\al,\be)$,
taking the supremum over test functions $h\in\Ff_b(\Xx,\R)$
defines a \emph{convex} divergence:
\begin{prop}\label{prop:KL_cvx}
The $\KL$ divergence is a (jointly) convex
function on $\Mm^+(\Xx)\times\Mm^+(\Xx)$.
\end{prop}

Going further, the density of continuous functions in the space of bounded
measurable functions allows us to restrict the 
optimization domain:

\begin{prop} \label{dual-KL-continuous}
Under the same assumptions,
\begin{align}
\KL(\al,\be)~=~\sup_{h\in\Cc(\Xx,\R)}~~ \dotp{\al}{h}-\dotp{\be}{e^h-1} \label{eq:KL_dual_continuous}
\end{align}
where $\Cc(\Xx,\R)$ is the space of (bounded) continuous functions
on the compact set $\Xx$.
\end{prop}
\begin{proof}
Let $h = \sum_{i\in I} h_i \,\textbf{1}_{A_i}$ be a simple Borel
function on $\Xx$, and let us choose some error margin $\delta > 0$.
Since $\al$ and $\be$ are Radon measures, for any $i$ in the finite
set of indices $I$, there exists a compact
set $K_i$ and an open set $V_i$ such that $K_i\subset A_i\subset V_i$
and
\begin{align*}
\sum_{i\in I} \max[ \, \al(V_i\backslash K_i)  \,,\, \be(V_i\backslash K_i)   \, ]
\leqslant \delta.
\end{align*}
Moreover, for any $i\in I$, there exists a continuous function
$\phi_i$ such that $\textbf{1}_{K_i}\leqslant \phi_i\leqslant \textbf{1}_{V_i}$.
The continuous function $g = \sum_{i\in I} h_i \phi_i$
is then such that
\begin{align*}
|\dotp{\al}{g-h}| &\leqslant \|h\|_\infty\,\delta
&\text{and} & &
|\dotp{\be}{e^g-e^h}| &\leqslant \|e^h\|_\infty\,\delta
\end{align*}
so that
\begin{align*}
|\,(\dotp{\al}{h}-\dotp{\be}{e^h-1})
\,&-\,
( \dotp{\al}{g}-\dotp{\be}{e^g-1})\,|\\
~\leqslant~& (\|h\|_\infty+\|e^h\|_\infty)\,\delta.
\end{align*}
As we let our simple function approach any measurable
function in $\Ff_b(\Xx,\R)$, choosing $\delta$
arbitrarily small, we then get~\eqref{eq:KL_dual_continuous} 
through \eqref{eq:KL_dual_mesurable}.
\end{proof}

We can then show that the Kullback-Leibler divergence is
weakly lower semi-continuous:

\begin{prop}\label{prop:KL_lsc}
If $\al_n\rightharpoonup \al$ and $\be_n \rightharpoonup \be$
are weakly converging sequences in $\Mm^+(\Xx)$, we get
\begin{align*}
\liminf_{n\rightarrow+\infty} ~\KL(\al_n,\be_n) ~\geqslant~ \KL(\al,\be).
\end{align*}
\end{prop}
\begin{proof}
According to \eqref{eq:KL_dual_continuous},
the KL divergence is defined as a pointwise supremum
of weakly continuous applications
\begin{align*}
\phi_h~:~(\al,\be)~\mapsto~ \dotp{\al}{h}-\dotp{\be}{e^h-1},
\end{align*}
for $h\in \Cc(\Xx,\R)$. It is thus
lower semi-continuous for the convergence in law.
\end{proof}

\subsection{SoftMin Operator}

\begin{prop}[The SoftMin interpolates between a minimum and a sum]
	Under the assumptions of the definition \eqref{eq:defn-softmin}, we get that
	\begin{align*}
		\mine{x\sim \al} \phi(x) 
		&\xrightarrow{\epsilon\rightarrow 0} 
		\min_{ x \in \Supp(\al)} \phi(x) \\
		&\xrightarrow{\epsilon\rightarrow +\infty} 
		\qquad\langle\al,\phi\rangle.
	\end{align*}
	If $\phi$ and $\psi$ are two continuous functions 
	in $\Cc(\Xx)$ such that $\phi \leqslant \psi$,
	\begin{align}
		\mine{x\sim \al} \phi(x) 
		\leqslant 
		\mine{x\sim \al} \psi(x). \label{eq:smin_order}
	\end{align}
	Finally, if $K\in\R$ is constant with respect to $x$, we have that 
	\begin{align}
		\mine{x\sim \al} \big[ K+\phi(x)\big]
		~=&~K+\mine{x\sim \al} \big[ \phi(x)\big].
		\label{eq:smin_constant}
	\end{align}
\end{prop}

\begin{prop}[The SoftMin operator is continuous]\label{prop:smin_continuous}
	Let $(\al_n)$ be a sequence of probability 
	measures converging \emph{weakly} towards $\al$,
	and $(\phi_n)$ be a sequence of continuous functions
	that converges \emph{uniformly} towards $\phi$.
	Then, for $\epsilon>0$,
	the SoftMin of the values of $\phi_n$ on $\al_n$
	converges towards the SoftMin of the values of $\phi$
	on $\al$, i.e.
	\begin{align*}
		\big( \al_n \rightharpoonup\al
		&,\,
		 \phi_n \xrightarrow{\|\cdot\|_\infty}\phi
		\big)\nonumber\\
		&\Longrightarrow
		\mine{x\sim \al_n} \phi_n(x)
		\rightarrow
		\mine{x\sim \al} \phi(x).
	\end{align*}
\end{prop}

\section{Proofs}
\label{sec:ot}


\subsection{Dual Potentials}
\label{appendix:dual-pot}

We first state some important properties of solutions $(\f,\g)$ to the dual problem~\eqref{eq-dual}. Please note that these results hold under the assumption that $(\Xx,\d)$ is a compact metric space, endowed with a \emph{ground cost} function $\C:\Xx\times\Xx\rightarrow\R$ that is $\kappa$-Lipschitz with respect to both of its input variables.

The existence of an optimal pair $(\f,\g)$ of potentials that reaches the maximal value of the dual objective is proved using the contractance of the Sinkhorn map $\T$,
defined in \eqref{eq:sinkhorn-op}, for the Hilbert projective metric~\citep{franklin1989scaling}.

While optimal potentials are only defined $(\al,\be)$-a.e., as highlighted in Proposition~\ref{prop:dual-pot-properties}, they are extended to the whole domain $\Xx$ by imposing, similarly to the classical theory of OT~\cite[Remark 1.13]{santambrogio2015optimal}, that they satisfy 
\eql{\label{eq:dual_optimal_smin}
    \f= \T(\be,\g)
    \qandq
    \g= \T(\al,\f),
}
with $T$ defined in \eqref{eq:sinkhorn-op}.
We thus assume in the following that this condition holds. The following propositions studies the uniqueness and the smoothness (with respect to the spacial position and with respect to the input measures) of these functions $(\f,\g)$ defined on the whole space.

\begin{prop}[Uniqueness of the dual potentials up to an additive constant] 
	\label{prop:uniqueness_dual}
Let $(\f_0,\g_0)$ and $(\f_1,\g_1)$ be two optimal pairs of dual potentials for a problem $\W_\epsilon(\al,\be)$ that satisfy \eqref{eq:dual_optimal_smin}. Then, there exists a constant $K\in\R$ such that
\begin{align}
	\f_0= \f_1+K~\text{and}~\g_0 = \g_1 - K.
	\label{eq:uniqueness_dual}
\end{align}
\end{prop}

\begin{proof}
	For $t\in[0,1]$, let us define
	$\f_t = \f_0 + t(\f_1-\f_0)$,
	$\g_t = \g_0 + t(\g_1-\g_0)$ and
	\begin{align*}
	\phi(t)= 
	&\langle\al,\f_t\rangle
	+
	\langle\be,\g_t\rangle\\
	&-
	\epsilon\langle\al\otimes\be,
	\exp\big(\tfrac{1}{\epsilon}(\f_t\oplus \g_t - \C)\big)-1\rangle,
	\end{align*}
	the value of the dual objective between the two optimal pairs.
	As $\phi$ is a concave function bounded above by 
	$\phi(0)=\phi(1)=\W_\epsilon(\al,\be)$, it is \emph{constant} with respect to $t$.
	Hence, for all $t$ in $[0,1]$,
	\begin{align*}
		\hspace{-.3cm}
		0 &= \phi''(t)\\
		&=
		-\tfrac{1}{\epsilon}
		\langle \al\otimes\be,
		e^{(\f_t\oplus \g_t - \C)/\epsilon }
		((\f_1-\f_0)\oplus (\g_1-\g_0))^2\rangle.
	\end{align*}
	This is only possible if, $\al\otimes\be$-a.e. in $(x,y)$, 
	\begin{align*}
		\big(\f_1(x)-\f_0(x)+\g_1(y)&-\g_0(y)\big)^2=0,
	\end{align*}
	i.e. there exists a constant $K\in\R$ such that
	\begin{align*}
		& \f_1(x)-\f_0(x)= +K\quad\text{$\al$-a.e.} \\
		& \g_1(y)-\g_0(y)=-K\quad\text{$\be$-a.e.}
	\end{align*}
	As we extend the potentials through~\eqref{eq:dual_optimal_smin}, 
	the SoftMin operator commutes with the addition of $K$~\eqref{eq:smin_constant}
	and lets our result hold on the whole feature space.
\end{proof}

\begin{prop}[Lipschitz property]
The optimal potentials $(\ab)$ of the dual problem~\eqref{eq-dual}
are both $\kappa$-Lipschitz functions on the feature space $(\Xx,\d)$.
\end{prop}

\begin{proof}
According to~\eqref{eq:dual_optimal_smin}, 
$\f$ is a SoftMin combination of 
$\kappa$-Lipschitz functions of the variable $x$;
using the algebraic properties of the SoftMin operator
detailed in (\ref{eq:smin_order}-\ref{eq:smin_constant}), 
one can thus show that $\f$ is a $\kappa$-Lipschitz function
on the feature space.
The same argument holds for $\g$.
\end{proof}

\begin{prop}[The dual potentials vary continuously with the input measures] 
\label{lem-uniform-conv}
Let $\al_n \rightharpoonup \al$ and 
$\be_n \rightharpoonup \be$ be weakly converging 
sequences of measures in $\Mm_1^+(\Xx)$.
Given some arbitrary anchor point $x_o\in\Xx$,
let us denote by $(\f_n,\g_n)$ the (unique) sequence of optimal
potentials for $\W_\epsilon(\al_n,\be_n)$ such that $\f_n(x_o)=0$.

Then, $\f_n$ and $\g_n$ converge uniformly towards the 
unique pair of optimal potentials $(\ab)$ for $\W_\epsilon(\al,\be)$
such that $\f(x_o)=0$.
Up to the value at the anchor point $x_o$, we thus have that
\begin{align*}
\big( \al_n \rightharpoonup \al, \,
\be_n \rightharpoonup \be
\big)\Longrightarrow
\big( \f_n \xrightarrow{\|\cdot\|_\infty} \f, \,
\g_n \xrightarrow{\|\cdot\|_\infty} \g
\big).
\end{align*}
\end{prop}

\begin{proof}
	For all $n$ in $\mathbb{N}$, the potentials $\f_n$ and $\g_n$
	are $\kappa$-Lipschitz functions on the compact, bounded set $\Xx$.
	As $\f_n(x_o)$ is set to zero, we can bound $|\f_n|$ on $\Xx$
	by $\kappa$ times the diameter of $\Xx$;
	combining this with~\eqref{eq:dual_optimal_smin},
	we can then produce a uniform bound on both $\f_n$ and $\g_n$:
	there exists a constant $M\in\R$ such that 
	\begin{align*}
		\foralls n\in\mathbb{N}, \foralls x\in\Xx, 
		-M\leqslant \f_n(x), \g_n(x) \leqslant+M.
	\end{align*}
	
	Being equicontinuous and uniformly bounded on the compact set $\Xx$,
	the sequence $(\f_n,\g_n)_n$ satisfies the hypotheses of the
	Ascoli-Arzela theorem: there exists a subsequence $(\f_{n_k},\g_{n_k})_k$
	that converges uniformly towards a pair $(\ab)$ of continuous functions.
	 $k$ tend to infinity, we see that $\f(x_o)=0$ and,
	using the continuity of the SoftMin operator (Proposition~\ref{prop:smin_continuous})
	on the optimality equations~\eqref{eq:optim-conditions-dual},
	we show that $(\ab)$ is an optimal pair for $\W_\epsilon(\al,\be)$.

	Now, according to Proposition~\ref{prop:uniqueness_dual},
	such a limit pair of optimal potentials $(\ab)$ is \emph{unique}.
	$(\f_n,\g_n)_n$ is thus a \emph{compact} sequence with a 
	\emph{single} possible adherence value:
	it has to converge, uniformly, towards $(\ab)$.
\end{proof}

\subsection{Proof of Proposition~\ref{prop:differentiability-ot}}
\label{appendix:prop:differentiability-ot}

The proof is mainly inspired from~\cite[Proposition 7.17]{santambrogio2015optimal}. 
Let us consider $\al$, $\delta\al$, $\be$, $\delta\be$
and times $t$ in a neighborhood of $0$, as in the statement above.
We define $\al_t=\al+t\delta\al$,
$\be_t=\be+t\delta\be$ and the variation ratio $\De_t$
given by
\begin{align*}
	\De_t \eqdef \frac{\W_\epsilon(\al_t,\be_t) - \W_\epsilon(\al,\be)}{t}.
\end{align*}
Using the very definition of $\W_\epsilon$
and the continuity property of Proposition~\ref{lem-uniform-conv},
we now provide lower and upper bounds on $\De_t$
as $t$ goes to $0$.

\paragraph{Weak$^*$ continuity.}

As written in~\eqref{eq:strong_duality},  $\W_\epsilon(\al,\be)$ can be computed through a straightforward, \emph{continuous} expression that does not depend on the value of the optimal dual potentials $(\ab)$ at the anchor point~$x_o$:
\begin{align*}
		\W_\epsilon(\al,\be)=
		\langle\al,\f\rangle+
		\langle\be,\g\rangle.
\end{align*}
Combining this equation with Proposition~\ref{lem-uniform-conv} (that guarantees the \emph{uniform} convergence of potentials  for weakly converging sequences of probability measures) allows us to conclude.

\paragraph{Lower bound.}

First, let us remark that $(\ab)$ is a
\emph{suboptimal} pair of dual potentials for
$\W_\epsilon(\al_t,\be_t)$.
Hence,
\begin{align*}
\W_\epsilon(\al_t,\be_t)
&\geqslant
\langle\al_t,\f\rangle
+
\langle\be_t,\g\rangle \\
&-
\epsilon\langle \al_t\otimes\be_t
, \exp\big( \tfrac{1}{\epsilon}( \f\oplus \g - \C) \big)-1\rangle
\nonumber
\end{align*}
and thus, since 
\begin{align*}
 \W_\epsilon(\al,\be) &= \langle\al,\f\rangle
+\langle\be,\g\rangle\\
&-\epsilon \langle\al\otimes\be, \exp(\tfrac{1}{\epsilon}(\f\oplus \g-\C))-1\rangle,
\end{align*}
one has
\begin{align*}
\De_t &\geqslant
\langle\delta\al,\f\rangle
+
\langle\delta\be,\g\rangle\\
&- \epsilon 
\langle \delta\al\otimes\be +
\al\otimes\delta\be , \exp(\tfrac{1}{\epsilon}(\f\oplus \g-\C))
\rangle+o(1) \nonumber \\
&\geqslant
\langle\delta\al,\f-\epsilon\rangle
+
\langle\delta\be,\g-\epsilon\rangle+o(1),
\end{align*}
since $\g$ and $\f$ satisfy the optimality 
equations~\eqref{eq:optim-conditions-dual}.

\paragraph{Upper bound.}

Conversely, let us denote by $(\g_t,\f_t)$ the optimal pair 
of potentials for $\W_\epsilon(\al_t,\be_t)$
satisfying $\g_t(x_o)=0$ for some arbitrary
anchor point $x_o\in\Xx$.
As $(\f_t,\g_t)$ are suboptimal potentials for $\W_\epsilon(\al,\be)$,
we get that
\begin{align*}
\W_\epsilon(\al,\be)
&\geqslant
\langle\al ,\f_t\rangle
+
\langle\be,\g_t\rangle \\
&-
\epsilon\langle \al\otimes\be
, \exp\big( \tfrac{1}{\epsilon}( \f_t\oplus \g_t - \C) \big)-1\rangle
\nonumber
\end{align*}
and thus, since 
\begin{align*}
 \W_\epsilon(\al_t,\be_t) &= \langle\al_t,\f_t\rangle+\langle\be_t,\g_t\rangle\\
 &-\epsilon \langle\al_t\otimes\be_t, \exp(\tfrac{1}{\epsilon}(\f_t\oplus \g_t-\C))-1\rangle,
\end{align*}
\begin{align*}
\De_t &\leqslant
\langle\delta\al,\f_t\rangle
+
\langle\delta\be,\g_t\rangle\\
&- \epsilon 
\langle \delta\al\otimes\be_t +
\al_t\otimes\delta\be 
, \exp(\tfrac{1}{\epsilon}(\f_t\oplus \g_t-\C))
\rangle+o(1) \nonumber \\
&\leqslant
\langle\delta\al,\f_t-\epsilon\rangle
+
\langle\delta\be,\g_t-\epsilon\rangle+o(1).
\end{align*}

\paragraph{Conclusion.}

Now, let us remark that as $t$ goes to $0$
\begin{align*}
\al+t\delta\al &\rightharpoonup \al
& \text{and}&
&
\be+t\delta\be &\rightharpoonup \be.
\end{align*}
Thanks to Proposition~\ref{lem-uniform-conv},
we thus know that $\f_t$ and $\g_t$ converge uniformly
towards $\f$ and $\g$.
Combining the lower and upper bound, we get
\begin{align*}
\De_t \xrightarrow{t\rightarrow0}
\langle\delta\al,\f-\epsilon\rangle
+
\langle\delta\be,\g-\epsilon\rangle
=
\langle\delta\al,\f\rangle
+
\langle\delta\be,\g\rangle,
\end{align*}
since $\delta\al$ and $\delta\be$ both have
an overall mass that sums up to zero.

\subsection{Proof of Proposition~\ref{prop:chgt_variable}}
\label{appendix:prop:chgt_variable}



The definition of $\W_\epsilon(\al,\al)$ is that
\begin{align*}
\W_\epsilon(\al,\al)
=
\max_{(\ab)\in \Cc(\Xx)^2}
&\langle\al,\f+\g\rangle\\
&-\epsilon
\langle \al\otimes\al,
e^{(\f\oplus \g - \C )/\epsilon}-1 \rangle.
\end{align*}

\paragraph{Reduction of the problem.}

Thanks to the symmetry of this concave problem
with respect to the variables $\f$ and $\g$,
we know that there exists a pair $(\f,\g=\f)$
of optimal potentials on the diagonal, and
\begin{alignat*}{8}
\W_\epsilon(\al,\al)
=
&\max_{\f\in \Cc(\Xx)}
2 \langle\al,\f\rangle\\
&\quad -\epsilon 
\langle \al\otimes\al,
e^{(\f\oplus \f - \C )/\epsilon}-1 \rangle.
\end{alignat*}
Thanks to the density of continuous
functions in the set of simple measurable functions,
just as in the proof of Proposition~\ref{dual-KL-continuous},
we show that this maximization can be done
in the full set of measurable functions
$\Ff_b(\Xx,\R)$:
\begin{alignat*}{8}
\W_\epsilon(\al,\al)
=
&\max_{\f\in \Ff_b(\Xx,\R)}
2 \langle\al,\f\rangle\\
&\quad -\epsilon 
\langle \al\otimes\al,
e^{(\f\oplus \f - \C )/\epsilon}-1 \rangle \\
= 
&\max_{\f\in \Ff_b(\Xx,\R)}
 2 \langle\al,\f\rangle\\
&\quad -\epsilon
\langle \exp(\f/\epsilon)\al, 
k_\epsilon\star \exp(\f/\epsilon)\al\rangle
+\epsilon,
\end{alignat*}
where $\star$ denotes the smoothing (convolution) operator
defined through
\begin{align*}
[k\star \mu](x)~=~ \int_\Xx k(x,y)\,\d\mu(y)
\end{align*}
for $k\in\Cc(\Xx\times \Xx)$ and $\mu\in \Mm^+(\Xx)$.

\paragraph{Optimizing on measures.}

Through a change of variables
\begin{align*}
\mu&=\exp(\f/\epsilon)\,\al
&\text{i.e.}& &
\f&= \epsilon\log\tfrac{\d\mu}{\d\al},
\end{align*}
keeping in mind that
$\al$ is a probability measure,
we then get that
\begin{alignat*}{8}
\W_\epsilon(\al,\al)
=
\epsilon\max_{\mu\in\Mm^+(\Xx),\al\ll\mu\ll\al}
&2 \langle\al,\log \tfrac{\d\mu}{\d\al}\rangle
\\- 
&\langle \mu, 
k_\epsilon\star \mu \rangle
+1\\
-\tfrac{1}{2}\W_\epsilon(\al,\al)
=
\epsilon\min_{\mu\in\Mm^+(\Xx),\al\ll\mu\ll\al}
&\langle\al,\log \tfrac{\d\al}{\d\mu}\rangle
\\ +\tfrac{1}{2}
&\langle \mu,
k_\epsilon\star \mu \rangle
-\tfrac{1}{2},
\end{alignat*}
where we optimize on positive measures $\mu \in\Mm^+(\Xx)$
such that $\al \ll \mu$ \text{and} $\mu \ll \al$.

\paragraph{Expansion of the problem.}

As $k_\epsilon(x,y)=\exp(-\C(x,y)/\epsilon)$ 
is positive for all $x$ and $y$ in $\Xx$,
we can remove the $\mu\ll\al$ constraint from the optimization problem:
\begin{alignat*}{8}-\tfrac{1}{2}\W_\epsilon(\al,\al)=
\epsilon\min_{\mu\in\Mm^+(\Xx),\al\ll\mu}
&\langle\al,\log \tfrac{\d\al}{\d\mu}\rangle
\\+\tfrac{1}{2}
&\langle \mu, 
k_\epsilon\star \mu \rangle
-\tfrac{1}{2}.
\end{alignat*}
Indeed, restricting a positive measure $\mu$ 
to the support of $\al$ lowers the right-hand term
$\langle\mu,k_\epsilon\star\mu\rangle$
without having any influence on the density 
of $\al$ with respect to $\mu$.
Finally, let us remark that the $\al\ll\mu$ constraint
is already encoded in the $\log\tfrac{\d\al}{\d\mu}$ operator,
which blows up to infinity if $\al$ has no density with respect to $\mu$;
all in all, we thus have:
\begin{alignat*}{8}
\F_\epsilon(\al)
&=-\tfrac{1}{2}\W_\epsilon(\al,\al)\\
&=
\epsilon\min_{\mu\in\Mm^+(\Xx)}
\langle\al,\log \tfrac{\d\al}{\d\mu}\rangle
+\tfrac{1}{2}
\langle \mu, 
k_\epsilon\star \mu \rangle-\tfrac{1}{2}
,
\end{alignat*}
which is the desired result.

\paragraph{Existence of the optimal measure $\mu$.}
In the expression above, the existence of an optimal $\mu$
is given as a consequence of the well-known fact
from OT theory that optimal dual potentials
$\f$ and $\g$ \emph{exist},
so that the dual OT problem \eqref{eq-dual}
is a max and not a mere supremum.
Nevertheless, since this property of $\F_\epsilon$
is key to the metrization of the convergence
in law by Sinkhorn divergences,
let us endow it with a direct,
alternate proof:

\begin{prop}\label{prop:Feps_unicity}
For any $\al\in\Mm_1^+(\Xx)$, assuming that $\Xx$ is compact,
there exists a unique $\mu_\al\in\Mm^+(\Xx)$ such that
\begin{align*}
\F_\epsilon(\al)~=~ \epsilon\,\big[\, 
\langle\al,\log \tfrac{\d\al}{\d\mu_\al}\rangle
+\tfrac{1}{2}
\langle \mu_\al, 
k_\epsilon\star \mu_\al \rangle-\tfrac{1}{2}
\,\big].
\end{align*}
Moreover, $\al \ll \mu_\al \ll \al$.
\end{prop}
\begin{proof}
Notice that for $(\al,\mu)\in\Mm_1^+(\Xx)\times\Mm^+(\Xx)$,
\begin{align*}
\E_\epsilon(\al,\mu)~&\eqdef~\langle\al,\log \tfrac{\d\al}{\d\mu}\rangle
+\tfrac{1}{2}
\langle \mu, 
k_\epsilon\star \mu \rangle \\
&=~ \KL(\al,\mu)~+~\dotp{\al-\mu}{1}~+~\tfrac{1}{2}\|\mu\|_{k_\epsilon}^2-\tfrac{1}{2}.
\end{align*}
Since $\C$ is bounded on the compact set $\Xx\times\Xx$ and $\al$
is a probability measure, we can already say that
\begin{align*}
\tfrac{1}{\epsilon}\F_\epsilon(\al)
\leqslant \E_\epsilon(\al,\al)-\tfrac{1}{2}
=\tfrac{1}{2}\dotp{\al\otimes\al}{e^{-\C/\epsilon}}-\tfrac{1}{2} <+\infty.
\end{align*}

\paragraph{Upper bound on the mass of $\mu$.}
Since $\Xx\times\Xx$ is compact and $k_\epsilon(x,y)>0$,
there exists $\eta > 0$ such that
$k(x,y)>\eta$ for all $x$ and $y$ in $\Xx$.
We thus get
\begin{align*}
\|\mu\|_{k_\epsilon}^2 ~\geqslant~ \dotp{\mu}{1}^2\,\eta
\end{align*}
and show that
\begin{align*}
\E_\epsilon(\al,\mu) ~\geqslant~& \dotp{\al-\mu}{1}~+~\tfrac{1}{2}\|\mu\|_{k_\epsilon}^2-\tfrac{1}{2}\\
~\geqslant~& \dotp{\mu}{1} ( \dotp{\mu}{1}\,\eta~-~1)-\tfrac{1}{2}.
\end{align*}
As we build a minimizing sequence $(\mu_n)$
for $\F_\epsilon(\al)$, we can thus assume that
$\dotp{\mu_n}{1}$ is uniformly bounded by some constant
$M>0$.

\paragraph{Weak continuity.}
Crucially, the Banach-Alaoglu theorem
asserts that
\begin{align*}
\{\, \mu\in\Mm^+(\Xx) ~|~ \dotp{\mu}{1}\leqslant M \,\}
\end{align*}
is weakly compact;
we can thus extract a weakly converging subsequence
$\mu_{n_k}\rightharpoonup \mu_\infty$ from
the minimizing sequence $(\mu_n)$.
Using Proposition~\ref{prop:KL_lsc}
and the fact that $k_\epsilon$ is continuous
on $\Xx\times\Xx$,
we show that $\mu\mapsto\E_\epsilon(\al,\mu)$ is a weakly
lower semi-continuous function:
$\mu_\infty=\mu_\al$ realizes the minimum
of $\E_\epsilon$ and we get our existence result.

\paragraph{Uniqueness.}
We assumed that our kernel $k_\epsilon$
is \emph{positive universal}.
The squared norm $\mu\mapsto\|\mu\|_{k_\epsilon}^2$
is thus a strictly convex functional
and using Proposition~\ref{prop:KL_cvx},
we can show that $\mu\mapsto \E_\epsilon(\al,\mu)$
is \emph{strictly} convex.
This ensures that $\mu_\al$ is uniquely defined.
\end{proof}

\subsection{Proof of Proposition~\ref{prop:concavity}}
\label{appendix:prop:concavity}


Let us take a pair of measures  $\al_0\neq\al_1$ in
$\Mm_1^+(\Xx)$,
and $t\in(0,1)$;
according to Proposition~\ref{prop:Feps_unicity},
there exists a pair of measures $\mu_0$, $\mu_1$ 
in $\Mm^+(\Xx)$ such that
\begin{align*}
(1-t)&\,\F_\epsilon(\al_0)
+ t\,\F_\epsilon(\al_1)\\
&=
\epsilon\,(1-t)\,\E_\epsilon(\al_0,\mu_0)
+\epsilon\,t\,\E_\epsilon(\al_1,\mu_1) \\
&>
\epsilon\,\E_\epsilon((1-t)\,\al_0 + t\,\al_1, (1-t)\,\mu_0 + t\,\mu_1) \\
&\geqslant
\F_\epsilon((1-t)\,\al_0 + t\,\al_1),
\end{align*}
which is enough to conclude.
To show the strict inequality, let us remark that
\begin{align*}
&(1-t)\,\E_\epsilon(\al_0,\mu_0)
+t\,\E_\epsilon(\al_1,\mu_1)\\
~=~&\E_\epsilon((1-t)\,\al_0 + t\,\al_1, (1-t)\,\mu_0 + t\,\mu_1)
\end{align*}
would imply that $\mu_0=\mu_1$,
since $\mu\mapsto\|\mu\|_{k_\epsilon}^2$
is strictly convex.
As $\al\mapsto\KL(\al,\be)$ is strictly
convex on the set of measures $\al$
that are absolutely continuous with respect to $\be$,
we would then have $\al_0=\al_1$ and
a contradiction with our first hypothesis.


\subsection{Proof of the Metrization of the Convergence in Law}
\label{appendix:metrize-conv-law}


The regularized OT cost is weakly continuous, and the uniform convergence
for dual potentials
ensures that $\H_\epsilon$ and
$\Wb_\epsilon$ are both continuous too.
Paired with~\eqref{eq:divergence_definite},
this property guarantees
the convergence towards $0$ of the Hausdorff and Sinkhorn divergences,
as soon as $\al_n \rightharpoonup\al$.

Conversely, let us assume that $\Wb_\epsilon(\al_n,\al)\rightarrow 0$ (resp. $\H_\epsilon(\al_n,\al)$).
Any weak limit $\al_{n_\infty}$
of a subsequence $(\al_{n_k})_k$ is equal to $\al$:
since our divergence is weakly continuous,
we have $\Wb_\epsilon(\al_{n_\infty},\al)=0$
(resp. $\H_\epsilon(\al_{n_\infty},\al)$),
and positive definiteness holds through~\eqref{eq:divergence_definite}.

In the meantime, since $\Xx$ is compact, 
the set of probability Radon measures $\Mm_1^+(\Xx)$ is 
sequentially compact for the weak-$\star$
topology.
$\al_n$ is thus a compact sequence with a unique adherence
value: it converges, towards $\al$.

\end{document}